\documentclass[letterpaper, 10 pt, conference]{ieeeconf}  
\usepackage{cite}
\usepackage{amsmath,amssymb,amsfonts}
\usepackage{algorithm}
\usepackage{algorithmic}
\usepackage{graphicx}        
\usepackage{subfigure}
\usepackage{enumerate}  
\usepackage{bm}
\newtheorem{assumption}{Assumption}
\newtheorem{theorem}{Theorem}
\newtheorem{lemma}{Lemma}

\newtheorem{remark}{Remark}                             

\IEEEoverridecommandlockouts                              
\title{\LARGE \bf
	Distributed Online Optimization with Coupled Inequality Constraints over Unbalanced Directed Networks
}

\author{Dandan Wang, Daokuan Zhu, Kin Cheong Sou and Jie Lu
	\thanks{D. Wang is with the School of Information Science and Technology, ShanghaiTech University, Shanghai 201210, China, with the University of Chinese Academy of Sciences, Beijing 100049, China, and with Shanghai Institute of Microsystem and Information Technology, Chinese Academy of Sciences, Shanghai 200050, China. Email:
		{\tt\small wangdd2@shanghaitech.edu.cn}. D. Zhu and J. Lu are with the School of Information Science and
		Technology, ShanghaiTech University, Shanghai 201210, China.
		Email:{\tt\small\{zhudk, lujie\}@shanghaitech.edu.cn}. K.C. Sou is with the Department
		of Electrical Engineering, National Sun Yat-sen University, Taiwan. Email: {\tt\small sou12@mail.nsysu.edu.tw}.}
}

\begin{document}

	\maketitle
	\thispagestyle{empty}
	\pagestyle{empty}

	\begin{abstract}
		
		This paper studies a distributed online convex optimization problem, where agents in an unbalanced network cooperatively minimize the sum of their time-varying local cost functions subject to a coupled inequality constraint. To solve this problem, we propose a distributed dual subgradient tracking algorithm, called DUST, which attempts to optimize a dual objective by means of tracking the primal constraint violations and integrating dual subgradient and push-sum techniques. Different from most existing works, we allow the underlying network to be unbalanced with a column stochastic mixing matrix. We show that DUST achieves sublinear dynamic regret and constraint violations, provided that the accumulated variation of the optimal sequence grows sublinearly. If the standard Slater’s condition is additionally imposed, DUST acquires a smaller constraint violation bound than the alternative existing methods applicable to unbalanced networks. Simulations on a plug-in electric vehicle charging problem demonstrate the superior convergence of DUST.		
	\end{abstract}

	\section{Introduction}
	
	Distributed online convex optimization (DOCO) has received considerable interest in recent years, motivated by its broad applications in dynamic networks with uncertainty, such as resource allocation for wireless network \cite{c1}, target tracking \cite{c2}, multi-robot surveillance \cite{c3}, and medical diagnosis\cite{c4}.
	In these scenarios, each agent in a network holds a time-varying local cost function and only has access to its real-time local cost function after making a decision based on historical information. Compared with centralized online optimization, DOCO enjoys prominent advantages in privacy protection, alleviation of computation and communication burden, and robustness to channel failures \cite{c5}.  
	
There has been a great number of distributed algorithms for solving DOCO problems \cite{c2,c3,c4,c6,c7,c8,c9,c10,c11,c12, c13, c14, c15}. Nevertheless, most of them are limited to unconstrained problems or simple set constraints, and do not allow for coupled inequality constraints that arise in many engineering applications. Coupled inequality constraints involve information from all agents, which poses a significant challenge to handle them in a distributed manner. To date, only a few distributed algorithms have been developed to address DOCO problems with coupled inequality constraints, including various variants of the saddle-point algorithm  \cite{c13,c10,c11,c12}, a primal-dual dynamic mirror descent algorithm \cite{c14} that has been extended to bandit settings in \cite{c15}, and a bandit distributed mirror descent push-sum algorithm \cite{c9}.	
	However, among these works, \cite{c15,c12,c13,c14} can only be applied to balanced networks with doubly stochastic mixing matrices. Although \cite{c9,c10,c11} consider unbalanced networks, their regret and constraint violation bounds are much worse than those in \cite{c15,c12,c13,c14}.
	
	In this paper, we focus on the DOCO problem with a coupled inequality constraint over an unbalanced network with a column stochastic mixing matrix, and propose a distributed dual subgradient tracking (DUST) algorithm to solve it. 
	To develop DUST, we first attempt to employ the subgradient method to address the dual problem of the constrained DOCO at each time instance. Then, to enable distributed implementation, we introduce auxiliary variables to track the primal constraint violations, which can be viewed as estimated dual subgradients. Finally, we harness the push-sum technique to eliminate the imbalance of networks, leading to the DUST algorithm. Our main contributions are elaborated as follows:	
	
	\begin{enumerate}
		\item DUST is able to address DOCO with \textit{coupled inequality} constraints over \textit{unbalanced} networks with column stochastic mixing matrices, while the alternative methods in \cite{c15,c12,c13,c14} require balanced interaction graphs. 
		\item We adopt \textit{dynamic regret} as the performance  measure of DUST, which is a more stringent metric than the static regret used in \cite{c13,c11,c12,c9}.
		\item
		We show that DUST achieves $\mathcal{O}(\sqrt{T}+V_T)$ dynamic regret and $\mathcal{O}(T^{\frac{3}{4}})$ constraint violation bounds, where $T$ is a finite time horizon and $V_T$ is the accumulated variation of the optimal sequence. Provided that $V_T$ grows sublinearly, DUST is able to achieve sublinear dynamic regret and constraint violations. Moreover, the constraint violation bound is improved to $\mathcal{O}(\sqrt{T})$ if we additionally assume the Slater's condition. To the best of our knowledge, there are no
		existing distributed algorithms achieving comparable dynamic regret and constraint violation bounds for DOCO problems with coupled constraints over unbalanced networks.
	\end{enumerate}
	
	The remainder of the paper is organized as follows. Section \ref{sec:problemformulation} formulates a DOCO with a coupled inequality constraint over unbalanced graphs with column stochastic mixing matrices. Section \ref{sec: DUSTalgorithm} develops the proposed DUST algorithm, and Section \ref{sec: regretanalysis} provides bounds of dynamic regret and constraint violations. Section \ref{sec: numericalexperiments} presents the numerical experiments, and Section \ref{sec:conclusion} concludes the paper.
	
	Throughout the paper, we adopt $\mathbb{R}^n$, $\mathbb{R}_+^n$ as the set of $n$-dimensional vectors, nonnegative vectors, respectively. For any set $X \subseteq {\mathbb{R}^n}$, $\operatorname{relint}(X)$ is its relative interior. $ A\otimes B$ represents the Kronecker product of any two matrices $A$ and $B$ with arbitrary size. Let $[a]_+$ represents the component-wise projection of a vector $a \in \mathbb{R}^n$ onto $\mathbb{R}_+^n$. Denote $I_d$ and $\mathbf{1}_p$ ($\mathbf{0}_p$) as the $d$-dimensional identity matrix and  the all-one (all-zero) column vectors with $p$ dimensions. Let $\|\cdot\|$ be the Euclidean norm.  $\langle x, y\rangle$ represents the standard inner product of two vectors $x$ and $y$. The notation $w_{ij,t}$ denotes the $i,j$-th component of matrix $W_t$ at time $t$. Let $\lceil \cdot \rceil$ and $\lfloor \cdot \rfloor$ be the ceiling and floor functions, respectively. For a convex function $f:\mathbb{R}^n\rightarrow\mathbb{R}$, 
	we denote $\partial f(x)$ as a subgradient of $f$ at $x$, i.e., $f(y) \ge f(x)+\langle\partial f(x), y-x\rangle$, $\forall y \in \mathbb{R}^n$.

	\section{Problem Formulation}\label{sec:problemformulation}
	Consider the network at each time $t\in \{1, \ldots, T\}$ modeled as a directed graph $\mathcal{G}_t=(\mathcal{V}, \mathcal{E}_t)$, where $\mathcal{V}=\{1,..., N\}$ is the set of nodes and $\mathcal{E}_t\subseteq\{\{i,j\}:i,j\in\mathcal{V},\;i\neq j\}$ is the set of edges. An edge $(j,i) \in  \mathcal{E}_t$ means  that node $i$ can receive a message from node $j$. Let $\mathcal{N}_{i,t}^{\text{in}}=\{j | (j,i) \in  \mathcal{E}_t\}\cup \{i\}$ and $\mathcal{N}_{i,t}^{\text{out}}=\{j|(i,j) \in  \mathcal{E}_t\}\cup\{i\}$ be the sets of in-neighbors and out-neighbors of node $i$, respectively. The mixing matrix $W_t$  associated with $\mathcal{G}_t$ is defined as $w_{ij,t} >0$ if $(j,i) \in  \mathcal{E}_t$ or $i=j$, and $w_{ij,t}=0$, otherwise. We assume each node $j \in \mathcal{V}$ only knows the weights related to its out-neighbors, i.e., $w_{ij,t}$, $\forall i \in \mathcal{N}_{j,t}^{\text{out}}$. We impose the following assumption on the interaction graph.
	\begin{assumption}\label{ass: networkassumption}
		 $\{\mathcal{G}_t\}_{t=1}^{T}$ satisfies:
		 	\begin{enumerate}
			\item There exists a constant $a \in (0,1)$ such that for each $t \ge 1$, $w_{ij,t}>a$ if $w_{ij,t} >0$. \label{graphnondegeneracy}
			\item For each $t \ge 1$, the mixing matrix $W_t$ is column stochastic, i.e., $\sum_{i=1}^{N} w_{ij,t}=1$ for all $j \in \mathcal{V}$.\label{columnstochastic}
			\item There exists an integer $B >0$ such that for any $k \ge 0$, the graph $(\mathcal{V}, \bigcup_{t=kB+1}^{(k+1)B}\mathcal{E}_t)$ is connected. \label{Bconnect}
		\end{enumerate}
	\end{assumption}
	
	An example of the mixing matrix that satisfies Assumption~\ref{ass: networkassumption} is $w_{ij,t}=1/d_{j,t}$, if $ i \in \mathcal{N}_{j,t}^{\text{out}}$; otherwise, $w_{ij,t}=0$, where 
	$d_{j,t}=|\mathcal{N}_{j,t}^{\text{out}}|$ is the out-degree of node $j$ at each time $t$. In this case, each node only needs to know its out-degree at each time $t$. Assumption~\ref{ass: networkassumption} ensures that there exists a path from one node to every other nodes within the interval of length $B$. Assumption~\ref{ass: networkassumption} is less restrictive than those in \cite{c12,c13,c14, c15}, which require $W_t$ to be doubly stochastic.

	We consider the distributed online problem with a globally coupled inequality constraint over the directed graph $\mathcal{G}_t$, where each node $i \in \mathcal{V}$ privately holds a time-varying local cost function $f_{i,t}:\mathbb{R}^{d_i}\rightarrow \mathbb{R}$, a constraint function $g_{i}:\mathbb{R}^{d_i}\rightarrow \mathbb{R}^p$, and a constraint set $X_i\subseteq\mathbb{R}^{d_i}$. Let $\mathbf{x}=[(x_1)^T, \ldots, (x_N)^T]^T\in \mathbb{R}^{\sum_{i=1}^{N}d_i}$ and $X=X_1\times \cdots \times X_N$ be the Cartesian product of all the $X_i$'s. At each time $t$, all nodes cooperate to minimize the sum of local cost functions while satisfying a coupled inequality constraint and set constraints, which can be written as   
	\begin{equation} \label{eq:primalprob}
		\begin{array}{cl}
			\underset{\substack{x_i, \forall i \in \mathcal{V} }}{\operatorname{minimize}} ~\!\!\!& f_t(\mathbf{x}):=\sum_{i=1}^{N} f_{i,t}(x_{i})\\  
			{\operatorname{ subject~to }}\!\!\! & \sum_{i=1}^{N}g_{i}(x_{i}) \leq \mathbf{0}_p,\\\vspace{2mm} 
			{} & {x_i \in X_i,\;\forall i \in \mathcal{V},}
		\end{array}
	\end{equation} 
	where the feasible set $\mathcal{X}:=\{\mathbf{x}\in X | \sum_{i=1}^{N}g_{i}(x_{i}) \leq \mathbf{0}_p \}$ is assumed to be nonempty.
	Note that the local cost function $f_{i,t}$ is unrevealed to node $i$ until it makes its decision $x_{i,t}$ at time $t$. Since node $i$ cannot access $f_{i,t}$ in advance, it is unlikely to obtain the exact optimal solution of problem \eqref{eq:primalprob}. Thus, it is desirable to develop a distributed online algorithm that generates local decisions $x_{i,t}$, $i \in \mathcal{V}$ to track the optimal solution. We make the following assumption on problem~\eqref{eq:primalprob}.
	
	\begin{assumption}\label{asm:primalcouplingprob}
		For any $t \ge 1$, problem~\eqref{eq:primalprob} satisfies 
		\begin{enumerate}
			\item For each $i\in\mathcal{V}$, $X_i$ is a compact convex set with diameter $R:= \sup_{x_i,\tilde{x}_i \in X_i} \|x_i-\tilde{x}_i\|$. 
			\item For each $i\in\mathcal{V}$, $ f_{i,t}$ and $g_i$ are convex on $X_{i}$.\label{couplinggisumficonvex}
		\end{enumerate}
	\end{assumption}

	It is directly obtained from the compactness of $X_i$'s and the convexity of $f_i$, $g_i$ in Assumption~\ref{asm:primalcouplingprob} that there exist constants $F > 0$,  $G>0$ such that 
	\begin{align}
		\|g_i(x_i)\|&\le F,\;\forall x_i \in X_i, \forall i \in \mathcal{V}, \label{eq:bounofgivalue}\\
	\|\partial  f_{i,t}(x_{i})\|&\le G,~\|\partial g_i(x_i)\|\le G, \;\forall x_i \in X_i, \forall i \in \mathcal{V}. \label{eq:boundsibradientoffigi}
	\end{align}

	We adopt \textit{dynamic regret} to measure the algorithm performance over a finite time horizon $T$ \cite{c2}, which is defined
	\begin{align}\label{definitionofdynamicregret}
		\operatorname{Reg}(T):=\sum_{t=1}^T \sum_{i=1}^N f_{i, t}\left(x_{i, t}\right)-\sum_{t=1}^T \sum_{i=1}^N f_{i, t}\left(x_{i, t}^*\right),
	\end{align}
	where $x_{i,t}^*$ is the $i$-th component of the optimal solution $\mathbf{x}_{t}^*=[(x_{1,t}^*)^T, \ldots, (x_{N,t}^*)^T]^T:=\underset{\mathbf{x} \in \mathcal{X}}{\arg \min } \sum_{i=1}^N f_{i, t}\left(x_i\right)$ to problem \eqref{eq:primalprob}. In contrast to the conventional metric \textit{static regret} that is defined as the difference between the accumulated cost over time and the cost incurred by the best fixed decision when all functions are known in hindsight (i.e., $\sum_{t=1}^T \sum_{i=1}^N f_{i, t}\left(x_{i}^*\right)$, where $\mathbf{x}^*=[(x_{1}^*)^T, \ldots, (x_{N}^*)^T]^T:=\underset{\mathbf{x} \in \mathcal{X}}{\arg \min } \sum_{t=1}^T\sum_{i=1}^N f_{i, t}\left(x_i\right)$), the dynamic regret \eqref{definitionofdynamicregret} allows the best decisions to vary with time and is a more stringent and suitable benchmark to capture the algorithm performance on a time-varying optimization problem \cite{c2,c3,c14}.
	
	In addition, we define the cumulative constraint violation to measure whether the coupled inequality is satisfied in a longterm run as follows:
	\begin{align}
		\operatorname{Reg}^c(T):=\left\|\left[\sum_{t=1}^T \sum_{i=1}^N g_i\left(x_{i, t}\right)\right]_{+}\right\|. \label{eq:definitionofConstVio}
	\end{align}
	
	Our goal is to design a distributed algorithm for solving the online problem \eqref{eq:primalprob} over $\mathcal{G}_t$ with superior dynamic regret and cumulative constraint violation bounds.
	
	\section{Algorithm Development}\label{sec: DUSTalgorithm}
	In this section, we propose a distributed dual subgradient tracking method to solve the distributed online problem with a coupled inequality constraint described in Section~\ref{sec:problemformulation}. 
	
	First of all, let $L_{t}: \mathbb{R}^{\sum_{i=1}^{N}d_i} \times \mathbb{R}_{+}^{p} \rightarrow \mathbb{R}$ be the Lagrangian function associated with problem \eqref{eq:primalprob} at time $t$:
	\begin{align}
		L_{t}(\mathbf{x},\mu)=f_t(\mathbf{x})+\mu^T\sum_{i=1}^{N}g_{i}(x_{i}), \displaybreak[0] \label{eq:lagrangianfunt}
	\end{align}
	where $\mu \ge \mathbf{0}_p$ is the Lagrange multiplier. We denote the dual function at time $t$ as $ D_{t}(\mu):=\min_{\mathbf{x}} \{ L_{t}(\mathbf{x},\mu)\}$. The dual problem of problem \eqref{eq:primalprob} at time $t$ is  $\max_{\mu \ge \mathbf{0}_p} D_{t}(\mu)$. If we directly apply the dual subgradient method \cite{c25} to the online problem \eqref{eq:primalprob}, we obtain the following updates: For arbitrarily given $\mu_1 \ge \mathbf{0}_p $ and each $ t \ge 1$ \begin{align}
		\mathbf{x}_{t+1}&=\operatorname{\arg min}\{L_{t+1}(\mathbf{x},\mu_t)\}, \displaybreak[0]\label{eq:xt1dualascentupdate}\\
		\mu_{t+1}&=[\mu_t+\sum_{i=1}^{N}g_{i}(x_{i,t+1})]_+, \displaybreak[0]\label{eq:mut1dualascentupdate}
	\end{align}   
	where $\mathbf{x}_{t+1}=[(x_{1,t+1})^T, \ldots, (x_{N,t+1})^T]^T \in \mathbb{R}^{\sum_{i=1}^{N}d_i}$, $\mu_{t+1} \in \mathbb{R}_p$ can be viewed as an estimate of the optimal solution to problem (1) at time $t+1$ and an estimate of the optimal dual solution to the dual problem of problem \eqref{eq:primalprob} at time $t+1$, respectively. The updates \eqref{eq:xt1dualascentupdate}--\eqref{eq:mut1dualascentupdate} actually optimize the dual problem of problem \eqref{eq:primalprob} at time $t+1$, i.e, $\max_{\mu \ge \mathbf{0}_p} D_{t+1}(\mu)$, by applying the subgradient method to compute the estimate of optimal dual solution, i.e, $\mu_{t+1}$, based on the historical information $\mu_t$, where the subgradient of the dual function $D_{t+1}(\mu)$ at $\mu_t$ is $\sum_{i=1}^{N}g_{i}(x_{i,t+1})$ according to the update \eqref{eq:xt1dualascentupdate} and the Danskin’s theorem \cite{c25}.

	However, \eqref{eq:xt1dualascentupdate} and \eqref{eq:mut1dualascentupdate} cause two issues. First, we have no prior knowledge of $f_{t+1}$ when making decision $\mathbf{x}_{t+1}$. Second, the above updates require the global quantities $\mu_t$ and $\sum_{i=1}^{N}g_{i}(x_{i,t+1})$ at each time $t$, which cannot be executed in a distributed scenario. To overcome the two issues, we let $g(x)=[(g_1(x_1))^T, \ldots, (g_N(x_N))^T]^T \in \mathbb{R}^{Np}$, and construct the following algorithm: Given $\mathbf{x}_1 \in X$, $\mathbf{y}_1=g(\mathbf{x}_1)$, $\bm{\mu}_1=\mathbf{0}_{Np}$, for any $t \ge 1$,
	\begin{align}
	\mathbf{x}_{t+1}&=\operatorname{\arg min}\limits_{x \in X } \left\{ \alpha_t \partial f_{t}(\mathbf{x}_{t})^T\!(\mathbf{x}\!-\!\mathbf{x}_{t})\right.\displaybreak[0]\notag\\ 
	&~~~\left.+\langle (W_t\otimes I_p)\bm{\mu}_{t}, ~g(\mathbf{x})\rangle\!+\eta_t\|\mathbf{x}\!-\!\mathbf{x}_{t}\|^2\right\}, \displaybreak[0] \label{eq:updateofxt11} \\
	\mathbf{y}_{t+1}&=(W_t\otimes I_p)\mathbf{y}_{t}+g(\mathbf{x}_{t+1})-g(\mathbf{x}_{t}), \displaybreak[0] \label{eq:updateofyk1}\\
	\bm{\mu}_{t+1}&=[(W_t\otimes I_p)\bm{\mu}_{t}+\mathbf{y}_{t+1}]_+, \label{eq:updateofuk1}
\end{align}
where $\mathbf{y}_t=[(y_{1,t})^T, \ldots, (y_{N,t})^T]^T\in \mathbb{R}^{Np}$, $\bm{\mu}_t=[(\mu_{1,t})^T, \ldots, (\mu_{N,t})^T]^T\in \mathbb{R}^{Np}$, and $W_t$ is the mixing matrix at time $t$ described in Section~\ref{sec:problemformulation}. Here, the parameters $\alpha_t$ is used to balance the objective optimization and the penalty of constraint violations and $\eta_t$ is the stepsize. 

The above updates \eqref{eq:updateofxt11}--\eqref{eq:updateofuk1} are capable of handling the issues caused by \eqref{eq:xt1dualascentupdate}--\eqref{eq:mut1dualascentupdate}, and potentially solve the online problem \eqref{eq:primalprob}. Specifically, we estimate the unknown $f_{t+1}$ with the first-order approximation of $f_{t}$ at $\mathbf{x}_{t}$, i.e., $f_{t}(\mathbf{x}_{t})+\partial f_{t}(\mathbf{x}_{t})^T\!(\mathbf{x}\!-\!\mathbf{x}_{t})$. The proximal term $\eta_t\|\mathbf{x}-\mathbf{x}_{t}\|^2$ in \eqref{eq:updateofxt11} guarantees the unique existence of $x_{t+1}$. To understand the distributed implementation of  \eqref{eq:updateofxt11}--\eqref{eq:updateofuk1}, let $x_{i,t}$, $y_{i,t}$, and $\mu_{i,t} $ be the $i$-th blocks of $\mathbf{x}_t$, $\mathbf{y}_t$, and $\bm{\mu}_t$, respectively. We let each node $i \in \mathcal{V}$ maintain $x_{i,t}$, $y_{i,t}$, and $\mu_{i,t} $ at time $t$. The term $\langle (W_t\otimes I_p)\bm{\mu}_{t}, ~g(\mathbf{x})\rangle$ in \eqref{eq:updateofxt11} not only enables distributed computation of $\mathbf{x}_{t+1}$, where each $x_{i,t+1}$ updates only involving its local information and the information received from its in-neighbors but also approaches to $\mu_t^T\sum_{i=1}^{N}g_{i}(x_{i})$ used in \eqref{eq:xt1dualascentupdate} if $W_t$ satisfies row stochasticity and each $\mu_{i,t}$ reaches the same value $\mu_t$. Owning to the column stochasticity of $W_t$ and the initial condition, from \eqref{eq:updateofyk1}, it is easy to obtain the update of $y_{i,t+1}$ and $\sum_{i=1}^{N}y_{i,t}=\sum_{i=1}^{N}g(x_{i,t})$, $ \forall t \ge 1$, which implies the local variable $y_{i,t}$ can track the primal constraint violations $\sum_{i=1}^N g_i(x_{i,t})$ at each time $t$. Thus, at time $t+1$, $y_{i,t+1}$ tracks the primal constraint violations $\sum_{i=1}^N g_i(x_{i,t+1})$, which can be regarded as the estimated subgradient of the dual function $D_{t+1}(\mu)$ at $\mu_t$ in \eqref{eq:mut1dualascentupdate}.  
The variable $\mu_{i,t+1}$ is the estimate for node $i$ of the optimal dual solution to the dual problem of problem \eqref{eq:primalprob} at time $t+1$, which is similar to $\mu_{t+1}$ in \eqref{eq:mut1dualascentupdate} also estimating the optimal dual solution at time $t+1$. Thus, each node $i \in \mathcal{V}$ computes $\mu_{i,t+1}$ with the weighted $\mu_{j,t}$ received from its all in-neighbor $j$ that facilitates consensual $\mu_{i,t}$, $\forall i \in \mathcal{V}$, and with the estimated dual subgradient of the dual function $D_{t+1}(\mu)$ at $\mu_t$ like \eqref{eq:mut1dualascentupdate}, i.e., $\sum_{i=1}^N g_i(x_{i,t+1})$, which can be tracked by the local variable $y_{i,t+1}$. Consequently, \eqref{eq:updateofxt11}--\eqref{eq:updateofuk1} leads to a distributed dual subgradient tracking algorithm (DUST). However, \eqref{eq:updateofxt11}--\eqref{eq:updateofuk1} do not work over unbalanced
networks since the column stochasticity of $W_t$ causes $\mu_{i,t}$, $\forall i \in \mathcal{V}$ cannot reach an identical value as they should. 
 
 To cope with unbalanced graphs, we integrate the push-sum technique into \eqref{eq:updateofxt11}--\eqref{eq:updateofuk1} to eliminate the imbalance of interaction networks by dynamically
 constructing row-stochastic matrices.
  We still refer to the resulting algorithm as DUST whose distributed implementation is described below. 
	
Let each node $i \in \mathcal{V}$ maintain variables $c_{i, t}\in \mathbb{R}$ besides $x_{i,t}$, $y_{i,t}$, and $\mu_{i,t} $. The DUST algorithm is described as follows: Given $x_{i,1} \in X_i$, $y_{i,1}=g_i(x_{i,1})$,
$c_{i,1}=1$, $\mu_{i, 1}=\mathbf{0}_p$, $\forall i \in \mathcal{V}$, for any $t \ge 1$, each node $i \in \mathcal {V}$ updates
	\begin{align}
		c_{i, t+1}&=\sum_{j \in \mathcal{N}_{i,t}^{\text{in}}} w_{i j, t} c_{j, t}, ~\label{eq:updateofcouplingcit}\displaybreak[0]\\
		\lambda_{i,t+1}&=\frac{\sum_{j \in \mathcal{N}_{i,t}^{\text{in}}} w_{i j, t} \mu_{j, t}}{c_{i,t+1}},  \label{eq:updateofcouplinglambdait1}\displaybreak[0]\\
		x_{i,t+1}&=\operatorname{\arg min}\limits_{x_{i} \in X_i } \left\{ \alpha_t\partial f_{i,t}(x_{i,t})^T\!(x_i\!-\!x_{i,t})\right.\notag\displaybreak[0]\\ 
		&~~\left.+\langle \lambda_{i,t+1}, ~g_i(x_i)\rangle\!+\eta_t\|x_i\!-\!x_{i,t}\|^2\right\}, \label{eq:updateofcouplingxt1} \displaybreak[0]\\
		y_{i,t+1}&=\sum_{j \in \mathcal{N}_{i,t}^{\text{in}}} w_{i j, t} y_{j, t}+g_i(x_{i,t+1})-g_i(x_{i,t}),\label{eq:updateofcouplingyk}\displaybreak[0]\\
		\mu_{i,t+1}&=\big[ \sum_{j \in \mathcal{N}_{i,t}^{\text{in}}} w_{i j, t} \mu_{j, t}+y_{i,t+1}\big]_+, \label{eq:updateofcouplinguk}\displaybreak[0]
	\end{align}
where the initialization $\mu_{i, 1}=\mathbf{0}_p$ is simply set satisfying $\mu_{i, 1} \ge\mathbf{0}_p$, and $y_{i,1}=g_i(x_{i,1})$, $\forall i \in \mathcal {V}$ ensures that $y_{i,t+1}$ tracks the estimated dual subgradient at time $t+1$ in \eqref{eq:updateofcouplingyk}. The updates \eqref{eq:updateofcouplingcit}, \eqref{eq:updateofcouplinglambdait1}, \eqref{eq:updateofcouplingyk}, and \eqref{eq:updateofcouplinguk} require node $i$ to collect $ w_{i j, t} c_{j, t}$, $ w_{i j, t} \mu_{j, t}$, and $ w_{i j, t}y_{j, t}$ from its every in-neighbor $j \in \mathcal{N}_{i,t}^{\text{in}}$ and \eqref{eq:updateofcouplingxt1}--\eqref{eq:updateofcouplinguk} are obtained by rewriting \eqref{eq:updateofxt11}--\eqref{eq:updateofuk1}. Obviously, the above updates only needs communication between neighboring nodes. Algorithm~\ref{alg:TVAlgorithm} details all these actions taken by the nodes. Before executing Algorithm~\ref{alg:TVAlgorithm}, all nodes need to determine the values of parameter $\alpha_t$ and the stepsize $\eta_t$. They can be set as $\alpha_t=\sqrt{t}$ and $\eta_t=t$ according to the theoretical results in  Section~\ref{sec: regretanalysis}. Different from \cite{c13, c14} whose parameters related to the time horizon $T$, we allow $\alpha_t$ and $\eta_t$ to be time-varying without knowing $T$ in advance, which provides flexibility for deciding when to stop the proposed online algorithm.
\begin{algorithm} [t]
	\caption{DUST}
	\label{alg:TVAlgorithm}
	\begin{algorithmic}[1]
		\STATE {\textbf{Initialization}}: 
		\STATE Each node $i \in \mathcal{V}$ sets $x_{i,1} \in X_i$, 
		$c_{i,1}=1$, $\mu_{i, 1}=\mathbf{0}_p$, and $y_{i,1}=g_i(x_{i,1})$.
		\FOR {$t=1,2,\ldots,T$}
		\STATE 
		Each node $j \in \mathcal{V}$ sends its local information $ w_{i j, t} c_{j, t}$, $ w_{i j, t} \mu_{j, t}$, and $ w_{i j, t}y_{j, t}$ to every out-neighbor $i \in \mathcal{N}_{j,t}^{\text{out}}$. After receiving the information from its in-neighbor $j \in \mathcal{N}_{i,t}^{\text{in}}$, each node $i \in \mathcal{V}$ updates $c_{i,t+1}$ according to \eqref{eq:updateofcouplingcit} and then computes $\lambda_{i,t+1}$ according to \eqref{eq:updateofcouplinglambdait1}.
		\STATE Each node $i \in \mathcal{V}$ updates $x_{i,t+1}$ according to \eqref{eq:updateofcouplingxt1}.
		\STATE Each node $i \in \mathcal{V}$ updates $y_{i,t+1}$ according to \eqref{eq:updateofcouplingyk}.
		\STATE Each node $i \in \mathcal{V}$ updates $\mu_{i,t+1}$ according to \eqref{eq:updateofcouplinguk}.
		\ENDFOR
	\end{algorithmic}
\end{algorithm}

	\section{Dynamic Regret and Constraint Violation Bounds}\label{sec: regretanalysis}
	In this section, we provide the dynamic regret and constraint violation bounds of DUST. 
	
	We first present the following lemmas.
	\begin{lemma}\label{lem:sumyit=sumgit}
		Suppose Assumptions~\ref{ass: networkassumption} and \ref{asm:primalcouplingprob} hold. Then,
		for any $t \ge 1$, 
		\begin{align}
			\sum_{i=1}^{N} y_{i,t}&=\sum_{i=1}^N g_i(x_{i,t}),\displaybreak[0] \label{eq:sumyitsumgit}\\
			\|y_{i,t}\| &\le B_y, \label{eq:boundednessofyit}\displaybreak[0]
		\end{align}
		where $B_y=\frac{8N^2F\sqrt{p}}{r}(1+\frac{2}{1-\sigma})+(N+2)F$, $r:= \inf_{t=1,2,...}(\min_{i \in [N]}\{W_t\cdots W_1\mathbf{1}_N\}_i) $, and $\sigma \in (0,1)$ satisfying $r \geq \frac{1}{N^{N B}},~\sigma \leq\left(1-\frac{1}{N^{N B}}\right)^{\frac{1}{N B}}$.
	\end{lemma}
	
	Lemma~\ref{lem:sumyit=sumgit} shows that the local estimator $y_{i,t}$ tracks the the sum of local constraint function values at each time $t$. The proof of Lemma~\ref{lem:sumyit=sumgit} is similar to Lemma~1 in \cite{c12} and Lemma~4 in \cite{c10}, and we omit it here.
	
	\begin{lemma}\label{lem:dualconsesusboud}
		Suppose Assumptions~\ref{ass: networkassumption} and \ref{asm:primalcouplingprob} hold. Then,
		for any $t \ge 1$,
		\begin{align}
			\sum_{i=1}^{N}\|\bar{\mu}_t-\lambda_{i,t+1}\| &\le \frac{8N^2B_y\sqrt{p}}{r}\sum_{k=1}^{t}\sigma^{t-k},\displaybreak[0] \label{eq:sumNbarmutlambda}
		\end{align}
		where $\bar{\mu}_t=\frac{1}{N}\sum_{i=1}^{N}\mu_{i,t}$ and $r, \sigma$ are given in Lemma~\ref{lem:sumyit=sumgit}.
	\end{lemma}	
	
	Lemma~\ref{lem:dualconsesusboud} presents a bound on the consensus error of the dual variables whose proof refers to Lemma~1 in \cite{c16}. The results in Lemma~\ref{lem:sumyit=sumgit}--\ref{lem:dualconsesusboud} involve a number parameters of network, such as the number of nodes $N$ and network connectivity factor $B$, which eventually influence the dynamic regret and constraint violation bounds through the following lemma.
	
	\begin{lemma}\label{lem:mubart1normandmubartbound}
		Suppose Assumptions~\ref{ass: networkassumption} and \ref{asm:primalcouplingprob} hold. Then, 
		for any $t \ge 1$ and arbitrary $\tilde{x}_{i,t} \in X_i$, $i \in \mathcal{V}$, 
		\begin{align}
			&	\frac{N}{2}\|\bar{\mu}_{t+1}\|^2\!-\!\frac{N}{2}\|\bar{\mu}_{t}\|^2  \displaybreak[0] \notag\\
			&\le\!(\frac{N}{2}\!\!+\!\!\frac{N}{r})B_y^2\!+\!\frac{NG^2\alpha_t^2}{4\eta_t}\!+\!(2B_y\!\!+\!\!2F)\!\sum_{i=1}^{N}\|\bar{\mu}_t\!-\!\lambda_{i,t+1}\|\displaybreak[0]\notag\\
			&+\sum_{i=1}^N \alpha_t\partial f_{i,t}(x_{i,t})^T\!(\tilde{x}_{i,t}\!-\!x_{i,t})+\sum_{i=1}^N\langle \bar{\mu}_t, ~g_i(\tilde{x}_{i,t})\rangle\displaybreak[0] \notag\\ 
			&+\sum_{i=1}^N \eta_t(\|x_{i,t}\!-\tilde{x}_{i,t}\|^2\!\!-\!\|x_{i,t+1}\!-\!\tilde{x}_{i,t}\|^2).\displaybreak[0] \label{eq:barmut1barmutnormbound}
		\end{align}
	\end{lemma}	
	\begin{proof}
		See Appendix~\ref{prf:proofofbarmut1barmutbound}.
	\end{proof}
	
	Lemma~\ref{lem:mubart1normandmubartbound} establishes the relationship between the bound on dual variables and the first-order information of the local functions, where the former involves constraint violations and the latter is related to the dynamic regret bound. By choosing $\tilde{x}_{i,t}$ appropriately and utilizing the convexity of local functions as well as Lemmas~\ref{lem:sumyit=sumgit}--\ref{lem:dualconsesusboud}, we obtain the dynamic regret and constraint violation bounds from Lemma~\ref{lem:mubart1normandmubartbound}.
	
	\begin{theorem} \label{thm:TVdynamicregret}
		Suppose Assumptions~\ref{ass: networkassumption} and \ref{asm:primalcouplingprob} hold. If we set  
		\begin{equation} \label{eq:TVdynamicstepsizechosen}
			\alpha_t=\sqrt{t}, \; \eta_t=t,
		\end{equation} then for any $t \ge 1$, 
		\begin{align}
			\operatorname{Reg}(T)&=\mathcal{O}(\sqrt{T})+\mathcal{O}(V_T), \displaybreak[0] \label{eq:TVdynamicresult}
		\end{align}
		where $V_T:=\sum_{t=1}^T\sqrt{t}\sum_{i=1}^N\left\|x_{i, t+1}^*-x_{i, t}^*\right\|$ and $x_{i,t}^*$ is the $i$-th component of the optimal solution $\mathbf{x}_{t}^*:=\underset{\mathbf{x} \in \mathcal{X}}{\arg \min } \sum_{i=1}^N f_{i, t}\left(x_i\right)$ to problem \eqref{eq:primalprob}.
	\end{theorem}
	\begin{proof}
		See Appendix~\ref{prf:TVdynamicregret}.
	\end{proof}
	
	Theorem \ref{thm:TVdynamicregret} shows that the dynamic regret grows sublinearly with $T$ if the accumulated variation of the optimal sequence $V_T$ is sublinear, which requires the online problem \eqref{eq:primalprob} does not change too drastically. Intuitively, the sublinearity guarantees that $\operatorname{Reg}(T)/T$ converges to $0$ as $T$ goes to infinity. It should be noted that if $V_T=0$, the result reduces to the static regret that achieves an $\mathcal{O}(\sqrt{T})$ bound. In addition, Theorem~\ref{thm:TVdynamicregret} indicates that DUST has stronger results on other existing algorithms applicable to coupled inequality constraints. Specifically, compared with \cite{c9, c10} that are also applicable to unbalanced networks with column stochastic matrices, the static regret bound in \cite{c9} is strictly greater than $\mathcal{O}(\sqrt{T})$ and the dynamic regret bound in \cite{c10} is $\mathcal{O}(T^{\frac{1}{2}+2\kappa})+\mathcal{O}(V_T)$, $\kappa \in (0, 1/4)$ that is worse than ours. Moreover, \cite{c12} assumes the boundedness of $\mu_{i,t}$ while Theorem~\ref{thm:TVdynamicregret} does not. Though \cite{c11,c13,c14,c15} can also handle coupled inequality constraints, \cite{c11,c13,c14,c15} are only applied to balanced networks with doubly stochastic mixing matrices and or \cite{c11} only focus on the static regret, which is weaker than our result. The dynamic regret bounds in \cite{c14,c15} depend on the accumulated error of optimal sequence $\sqrt{T}\sum_{t=1}^T\sum_{i=1}^N\left\|x_{i, t+1}^*-x_{i, t}^*\right\|$, which is large than $V_T$ in \eqref{eq:TVdynamicresult}, leading to a larger bound than \eqref{eq:TVdynamicresult}. 
	
	Next, we present a bound on constraint violation. 
	
	\begin{theorem}\label{thm:WioutSlaterCV}
		Suppose all the conditions in Theorem~\ref{thm:TVdynamicregret} hold. Then for any $t \ge 1$,
		\begin{align}
			\operatorname{Reg}^c(T) =\mathcal{O}(T^{\frac{3}{4}}). \displaybreak[0] \label{eq:TVwithoutSlaterconstraintvio}
		\end{align}	
	\end{theorem}	
	\begin{proof}
		See Appendix~\ref{prf:WioutSlaterCV}.
	\end{proof}
	
	Theorem~\ref{thm:WioutSlaterCV} shows that DUST achieves a sublinear constraint violation bound. The result is superior than \cite{c9,c10, c11} whose constraint violation bound is strictly greater than $\mathcal{O}(T^{\frac{3}{4}})$. Theorem~2 holds without assuming the Slater's condition that allows us to handle equality constraints by converting an equality into two inequalities. The following theorem shows that $Reg^c(T)$ is improved to $\mathcal{O}(\sqrt{T})$ if all local constraint functions $g_i, \forall i \in \mathcal{V}$ satisfy the Slater's condition, which is commonly assumed in \cite{c10,c11,c14}.
	
	\begin{assumption} \label{asm: slatercondition}
		There exists a constant $\epsilon >0$ and a point $\hat{x}_i \in\operatorname{relint}(X_i),~\forall i \in \mathcal{V}$ such that $\sum_{i=1}^{N}g_{i}(\hat{x}_{i}) \leq-\epsilon\mathbf{1}_p$.
	\end{assumption}
	
	\begin{theorem} \label{thm:WiSlaterCV}
		Suppose Assumptions~\ref{ass: networkassumption}$-$\ref{asm: slatercondition} hold. If we set  $\eta_t$ and $V_t$ as these in Theorem~\ref{thm:TVdynamicregret}. Then, for any $t \ge 1$,
		\begin{align}
			\operatorname{Reg}^c(T) =\mathcal{O}(\sqrt{T}). \displaybreak[0] \label{eq:TVwithSlaterconstraintvio}
		\end{align}	
	\end{theorem}
	\begin{proof}
		See Appendix~\ref{prf:WiSlaterCV}.
	\end{proof}
	
	\begin{remark}\label{rmk:BNeffects}
		To the best of our knowledge, DUST is the first distributed algorithm achieving  $\mathcal{O}(\sqrt{T})$ dynamic regret bound and $\mathcal{O}(T^{\frac{3}{4}})$ constraint violation bound for DOCO problems with coupled inequality constraints over unbalanced networks, let alone achieving $\mathcal{O}(\sqrt{T})$ constraint violation bound. Unlike \cite{c20,c21,c22} whose constraint violation bounds are affected by the dynamic optimal decisions $x_t^*$, $\forall t \ge 1$, our results are independent of them. Furthermore, from Appendix~\ref{prf:TVdynamicregret}--\ref{prf:WiSlaterCV}, we observe that the bounds of $\operatorname{Reg}(T)$ and $\operatorname{Reg}^c(T)$
		in \eqref{eq:TVdynamicresult}$-$\eqref{eq:TVwithSlaterconstraintvio} are proportional to $\frac{N^6}{r^3(1-\sigma)^3}$ with $r \geq \frac{1}{N^{N B}},~\sigma \leq\left(1-\frac{1}{N^{N B}}\right)^{\frac{1}{N B}}$. 
		Note that $\frac{N^6}{r^3(1-\sigma)^3}$ increases as $N$ and $B$ grow and  $\operatorname{Reg}(T)$ and $\operatorname{Reg}^c(T)$ increase accordingly. This statement is verified via a numerical example in the following section.
	\end{remark}
	
	\section{Numerical Example} \label{sec: numericalexperiments}
		\begin{figure}[t] 	
		\subfigure[]
		{
			\begin{minipage}{0.466\linewidth}
				\centering 
				\includegraphics[height=3.47cm,width=4cm]{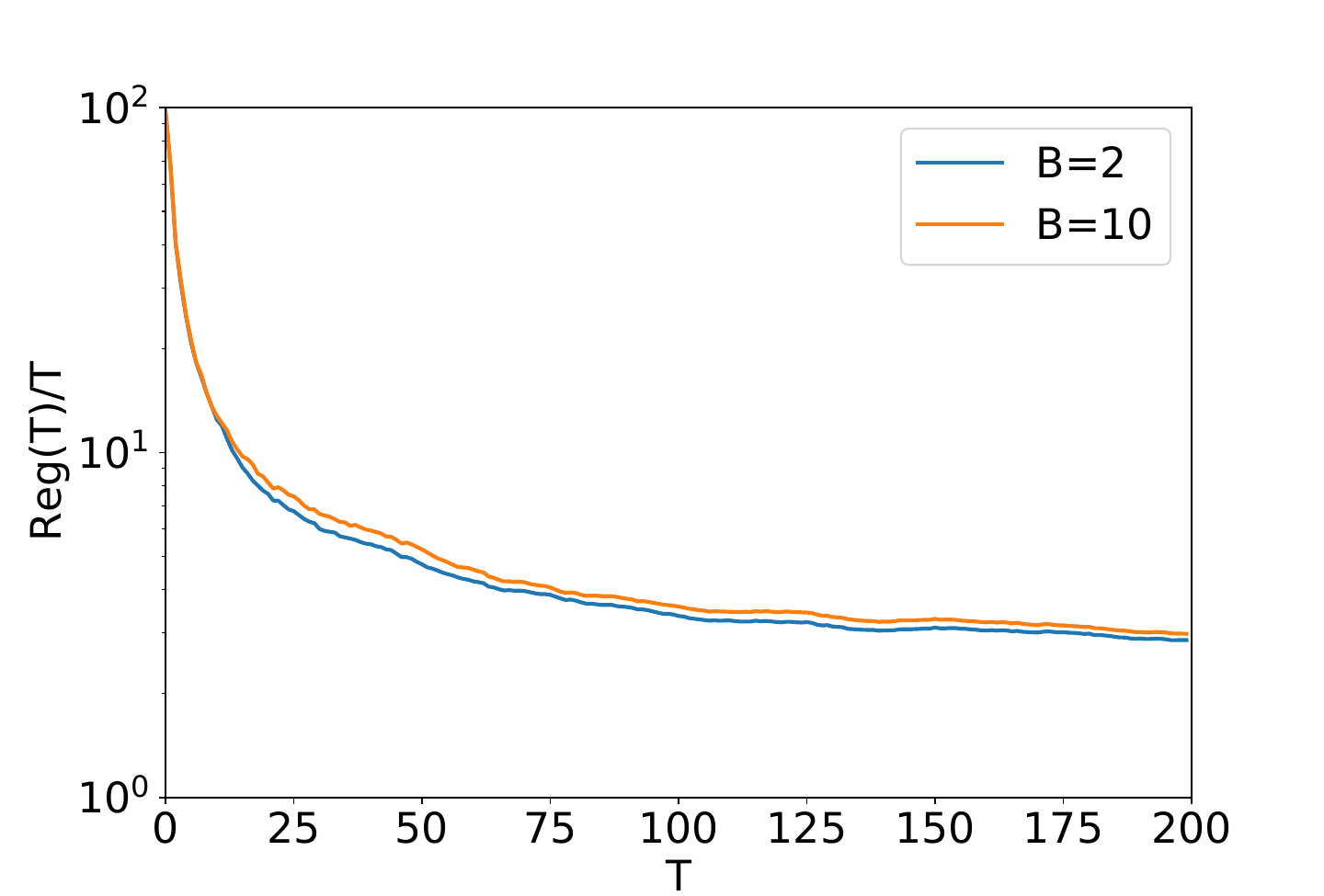}
				\label{cooperation}
			\end{minipage}
		}	
		\subfigure[]
		{
			\begin{minipage}{0.466\linewidth}
				\centering      
				\includegraphics[height=3.47cm,width=4cm]{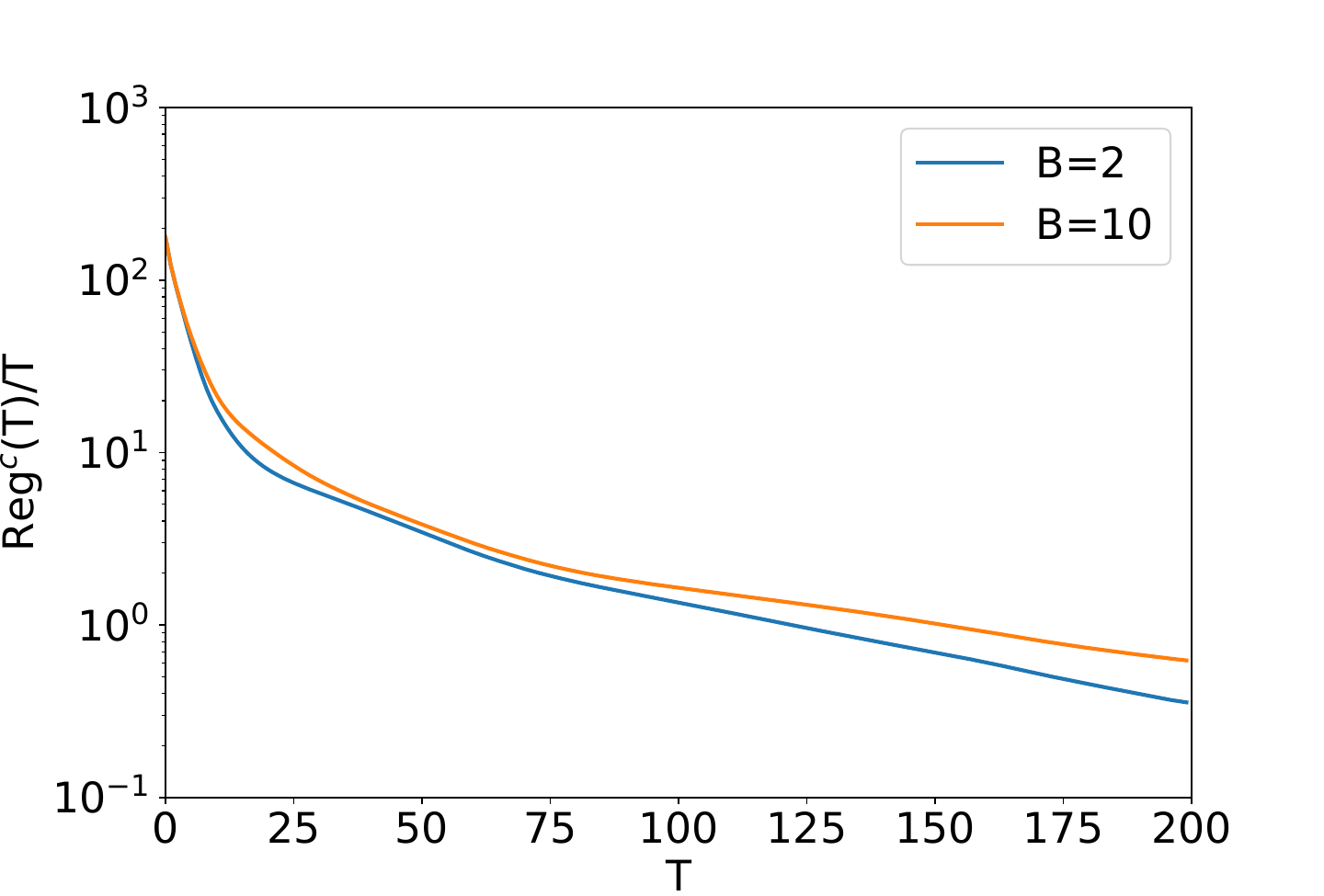}
				\label{eachlatency}
			\end{minipage}
		}	
		\caption{Effects of network connectivity factor $B$ on (a) $\operatorname{Reg}(T)/T$ and (b) $\operatorname{Reg}^c(T)/T$ when $N=10$.  } 
		\label{differentB}  
	\end{figure}
	%
	\begin{figure}[t] 	
		\subfigure[]
		{
			\begin{minipage}{0.466\linewidth}
				\centering 
				\includegraphics[height=3.47cm,width=4cm]{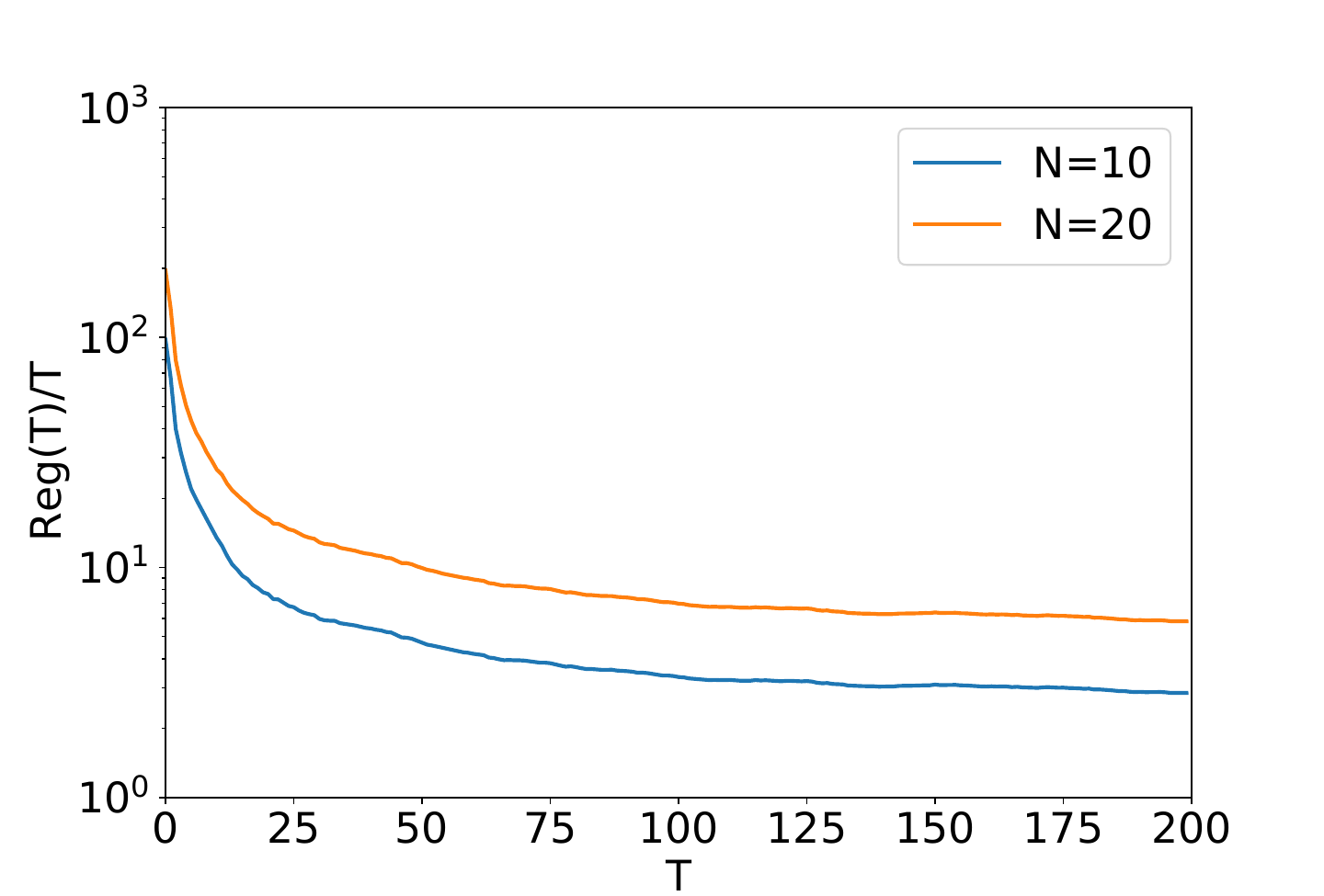}
				\label{regdifferentN}
			\end{minipage}
		}	
		\subfigure[]
		{
			\begin{minipage}{0.466\linewidth}
				\centering      
				\includegraphics[height=3.47cm,width=4cm]{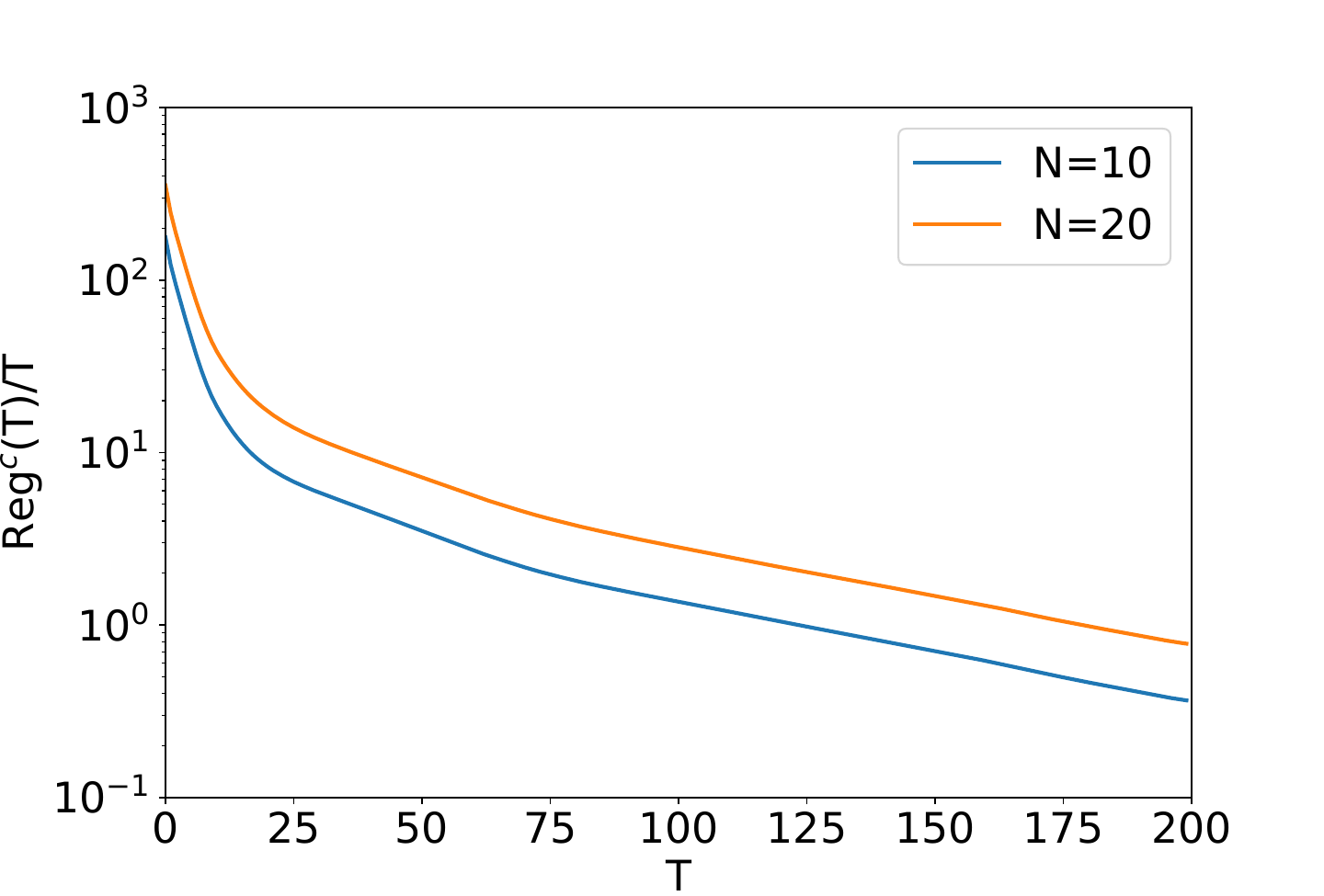}
				\label{cvdifferentN}
			\end{minipage}
		}	
		\caption{Effects of node number $N$ on (a) $\operatorname{Reg}(T)/T$ and (b) $\operatorname{Reg}^c(T)/T$ when $B=2$.  } 
		\label{differentN}  
	\end{figure}

	\begin{figure}[t] 	
		\subfigure[]
		{
			\begin{minipage}{0.466\linewidth}
				\centering 
				\includegraphics[height=3.47cm,width=4cm]{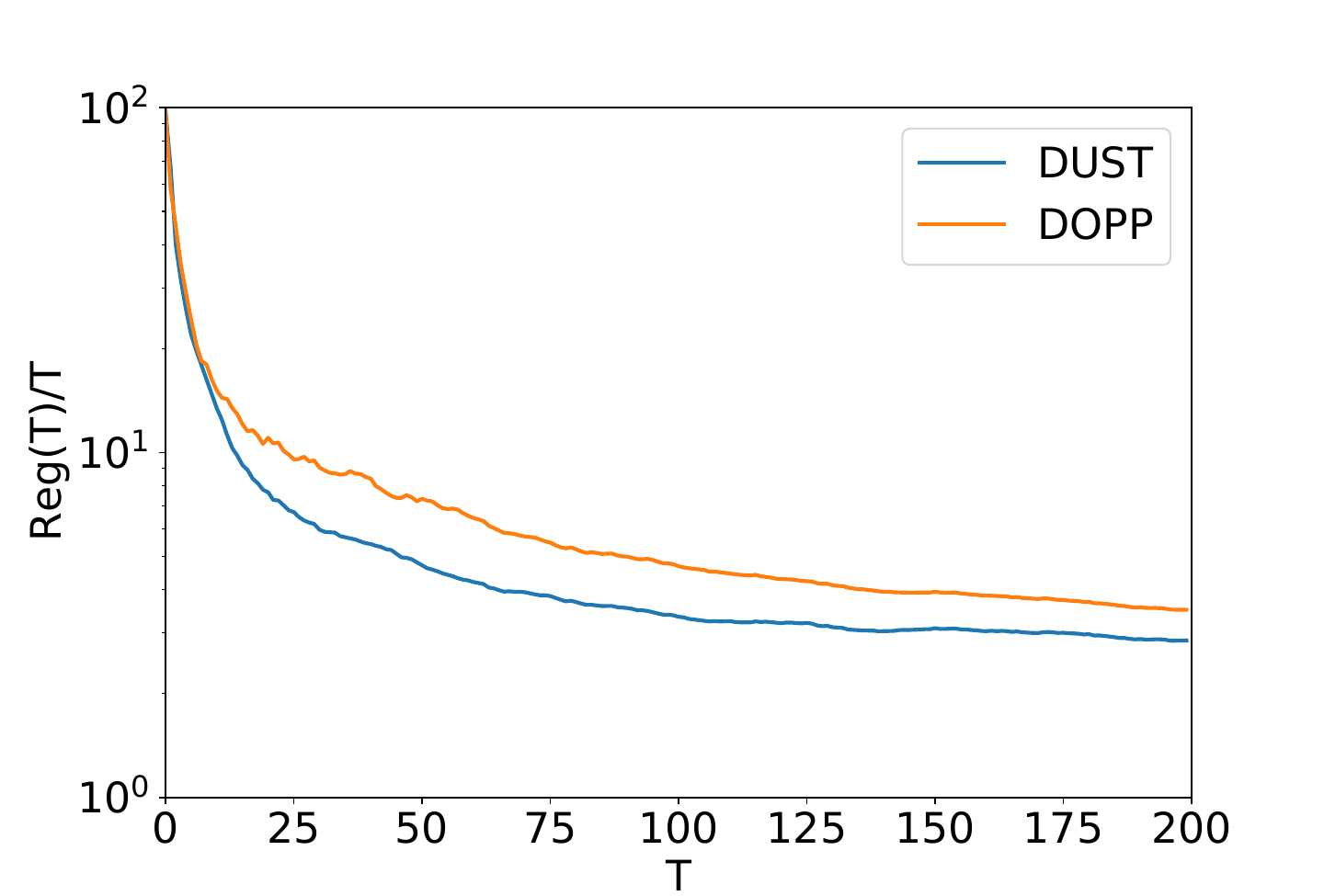}
				\label{regcompare}
			\end{minipage}
		}	
		\subfigure[]
		{
			\begin{minipage}{0.466\linewidth}
				\centering      
				\includegraphics[height=3.47cm,width=4cm]{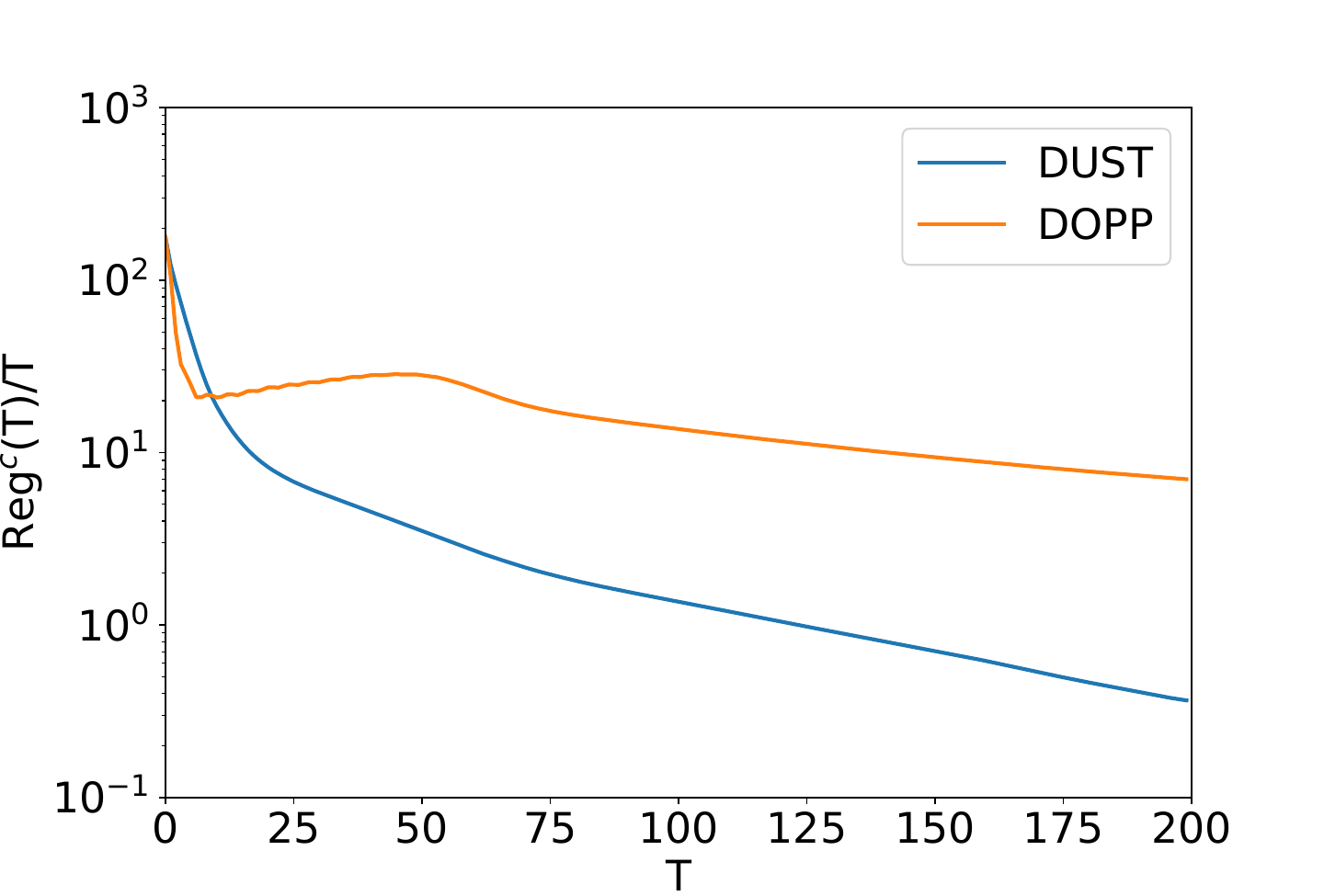}
				\label{cvcompare}
			\end{minipage}
		}	
		\caption{ Comparison between DUST and DOPP \cite{c10} on (a) $\operatorname{Reg}(T)/T$ and (b) $\operatorname{Reg}^c(T)/T$ when $N=10$, $B=4$.} 
		\label{compare}  
	\end{figure}
	We apply DUST to solve the plug-in electric vehicles
	(PEVs) charging problem whose goal is to find an optimal charging schedule over a time period by minimizing the sum of the time-varying local cost of all PEVs while satisfying certain constraints at each time instance\cite{c10,c13}. At each time $t$, the PEVs charging problem can be cast as: 
	\begin{equation} \label{eq:PEVsproblem}
		\begin{array}{cl}
			\underset{\substack{x_i \in X_i, \forall i \in \mathcal{V} }}{\operatorname{minimize}} ~\!\!\!& \sum_{i=1}^{N} c_{i,t}(x_{i})\\  
			{\operatorname{ subject~to }}\!\!\! & \sum_{i=1}^{N}A_{i}x_{i}-D/N \leq \mathbf{0}_p,
		\end{array}
	\end{equation} 
	where $x_i$ represents the charging rate of PEV $i$,  $c_{i,t}(x_{i}):=a_{i,t}/2\|x_i\|^2+b_{i,t}^Tx_i$ is local cost function of PEV $i$ at time $t$ \cite{c23}, and $X_i$ is the local constraint set involving maximum charging power and desired final state of charge of PEV $i$. The coupled constraint $\sum_{i=1}^{N}A_{i}x_{i}-D/N \leq \mathbf{0}_p$ guarantees that the aggregate charging power of all PEVs is less than maximum power that the network can deliver. In our simulation, each $a_{i,t}$ and $b_{i,t}$ are randomly generated from uniform distributions $[0.5, 1]$ and $\left(0, 1\right]^{d_i}$, respectively, where $d_i=24$ is the dimension of $x_i$.
	According to the set-up in \cite{c24}, there are 48 coupled inequalities, i.e., the rate aggregation
	matrix $A_i \in  \mathbb{R}^{48 \times 24}$ and each local set $X_i$ is determined by 197 inequalities. The values of $A_i$, $D$, and $X_i$ are obtained by referring to \cite{c24}. 
	

	To investigate the convergence performance of DUST and the effects of network connectivity factor $B$ and node number $N$ on the convergence performance of DUST, we run DUST with different $B$ and different $N$. Fig.~\ref{differentB} and Fig.~\ref{differentN} plot the evolution of $\operatorname{Reg}(T)/T$ and  $\operatorname{Reg}^c(T)/T$ with $B=2,10$ when $N$ is fixed as 10 and $N=10,20$ when $B$ is fixed as 2, respectively. From the two figures, we observe that DUST is able to achieve sublinear convergence in terms of dynamic regret and constraint violations, which validates our theoretical results. In addition,  it can be seen that the convergence speed becomes slower if $B$ or $N$ increases. This fact is consistent with our analysis in Remark~\ref{rmk:BNeffects}.
	

	We compare DUST with the distributed online primal-dual push-sum (DOPP) in \cite{c10}, which is also developed based on column stochastic mixing matrices. For a fair comparison, we set $\kappa=0.2$ for DOPP so that it achieves satisfactory convergence performance. Fig.~\ref{compare} presents the evolution of $\operatorname{Reg}(T)/T$ and  $\operatorname{Reg}^c(T)/T$ of DUST and DOPP with $N=10$, $B=4$. It is evident that DUST achieves smaller dynamic regret and constraint violations than DOPP and thus validates the superior performance of DUST.

	\section{Conclusion}\label{sec:conclusion}
	We have constructed a distributed dual subgradient tracking algorithm (DUST) to solve the DOCO problem with a globally coupled inequality constraint over unbalanced networks. To develop it, we integrate the push-sum technique into the dual subgradient method. The subgradients with respect to dual variables can be estimated by primal constraint violations, which is tracked by local auxiliary variables, enabling distributed implementation. We show that DUST achieves sublinear dynamic regret and constraint violations if the accumulated variation of the optimal sequence is also sublinear. Our theoretical results are stronger than those of existing distributed algorithms applicable to unbalanced networks, which is verified via numerical experiments. 

	\appendix
	
	For ease of exposition, let $\hat{\mu}_{i, t}=\sum_{j \in \mathcal{N}_{i,t}^{\text{in}}} w_{i j, t} \mu_{j, t}$, $\bar{\mu}_t=\frac{1}{N}\sum_{i=1}^{N}\mu_{i,t}$, and $\epsilon_{i,t+1}=[\hat{\mu}_{i, t}+y_{i,t+1}]_+-\hat{\mu}_{i, t}$. With these notations, we rewrite \eqref{eq:updateofcouplinglambdait1} and \eqref{eq:updateofcouplinguk} as
%

	\subsection{Proof of Lemma~\ref{lem:mubart1normandmubartbound}} \label{prf:proofofbarmut1barmutbound}
	Before we prove Lemma~\ref{lem:mubart1normandmubartbound}, the following auxiliary lemma is first presented. 
	\begin{lemma}\label{lem:epsionlboundcitbound}
		Suppose Assumptions~\ref{ass: networkassumption} and \ref{asm:primalcouplingprob} hold. Then
		\begin{align}
			\|\epsilon_{i,t+1}\|\le B_y,~r \le c_{i,t} \le N,~\forall t\ge 1, \displaybreak[0] \label{eq:boundofepsiloncit}
		\end{align}	
		where $B_y$ and $r$ are given in Lemma~\ref{lem:sumyit=sumgit}.
	\end{lemma}
	\begin{proof}
		According to the property of projection $\|P_S(z_1)-P_S(z_2)\| \le \|z_1-z_2\|$, $\forall z_1, z_2 \in \mathbb{R}^n$, we have $\|\epsilon_{i,t+1}\|=\|[\hat{\mu}_{i, t}+y_{i,t+1}]_+-\hat{\mu}_{i, t-1} \| \le \|\hat{\mu}_{i, t-1}+y_{i,t}-\hat{\mu}_{i, t-1}\|\le \|y_{i,t}\| \le B_y$. The proof of boundedness of $c_{i,t}$ can refer to Lemma~3 in \cite{c10}.
	\end{proof}
	
Summing	$\mu_{i, t+1}=\hat{\mu}_{i, t}+\epsilon_{i,t+1}$from $i=1$ to $N$ yields
	\begin{align}
		\bar{\mu}_{t+1}=\bar{\mu}_t+\frac{1}{N}\sum_{i=1}^{N}\epsilon_{i,t+1}, \label{eq:barmut1barmut}
	\end{align}
	which gives for all $\lambda \in \mathbb{R}_+^p$,
	\begin{align}
		\|\bar{\mu}_{t+1}\!-\!\lambda \|^2 \le \|\bar{\mu}_t\!-\!\lambda \|^2\!+\!\frac{2}{N}\sum_{i=1}^{N}\epsilon_{i,t+1}^T(\bar{\mu}_t\!-\!\lambda)\!+\!B_y^2.\label{eq:barmut1withepsilonbarmut}
	\end{align}
	The last inequality in \eqref{eq:barmut1withepsilonbarmut} follows from \eqref{eq:boundofepsiloncit}.
	Let us now consider the term $\sum_{i=1}^{N}\epsilon_{i,t+1}^T(\bar{\mu}_t-\lambda)$. It can be obtained	
	\begin{align}
		&\sum_{i=1}^{N}\!\epsilon_{i,t+1}^T(\bar{\mu}_t\!-\!\lambda)\!=\!\!\sum_{i=1}^{N}\!\epsilon_{i,t+1}^T(\frac{(\hat{\mu}_{i, t}\!-\!c_{i,t+1}\lambda)}{c_{i,t+1}}\!+\!\bar{\mu}_t\!-\!\lambda_{i,t+1})\displaybreak[0]\notag\\
		&=\!\sum_{i=1}^{N}\!\big(\epsilon_{i,t+1}\!\!-\!\!y_{i,t+1})^T\!\frac{(\mu_{i, t+1}\!\!-\!\!c_{i,t+1}\lambda)}{c_{i,t+1}}\!+\!\epsilon_{i,t+1}^T(\bar{\mu}_t\!\!-\!\!\lambda_{i,t+1})\displaybreak[0] \notag\\
		&+\!(\epsilon_{i,t+1}\!-\!y_{i,t+1})^T\frac{(\hat{\mu}_{i, t}\!-\!\mu_{i, t+1})}{c_{i,t+1}}\big)\!+\!y_{i,t+1}^T(\lambda_{i,t+1}\!-\!\lambda)\! \displaybreak[0]\notag\\
		&\le \!\sum_{i=1}^{N}\!y_{i,t+1}^T(\frac{(\mu_{i, t+1}\!\!-\!\!\hat{\mu}_{i, t})}{c_{i,t+1}}\!+\!\lambda_{i,t+1}\!\!-\!\!\lambda)+\!\epsilon_{i,t+1}^T(\bar{\mu}_t\!\!-\!\!\lambda_{i,t+1}) \displaybreak[0]\notag\\ 
		& \le \!\frac{N}{r}\!B_y^2\!+\!\sum_{i=1}^{N}\!y_{i,t+1}^T(\lambda_{i,t+1}\!\!-\!\!\lambda)\!+\!B_y\!\sum_{i=1}^{N}\!\|\bar{\mu}_t\!-\!\lambda_{i,t+1}\|, \displaybreak[0] \label{eq:epsilonmultiplybarmutlambdabound}
	\end{align}	
	where the first equality uses $\lambda_{i,t+1}=\frac{\hat{\mu}_{i, t}}{c_{i,t+1}}$. The first inequality uses: (a) $(\epsilon_{i,t+1}\!-y_{i,t+1})^T(\mu_{i, t+1}\!-\!c_{i,t+1}\lambda) \le 0$ according to the property of projection $(P_S(x)-x)^T(P_S(x)-y) \le 0$,  $\forall x \in \mathbb{R}^n, y \in S$, where $\lambda \in \mathbb{R}_+^p$; (b) $\epsilon_{i,t+1}^T(\hat{\mu}_{i, t}\!-\!\mu_{i, t+1})=(\mu_{i, t+1}-\hat{\mu}_{i, t})^T(\hat{\mu}_{i, t}\!-\!\mu_{i, t+1}) \le 0$. The last inequality uses: (a) the Cauchy–Schwarz inequality; (b) \eqref{eq:boundednessofyit} and \eqref{eq:boundofepsiloncit}.  
	The term $y_{i,t+1}^T(\lambda_{i,t+1}-\lambda)$ in \eqref{eq:epsilonmultiplybarmutlambdabound} can be obtained
	\begin{align}
		&\sum_{i=1}^{N}\!y_{i,t+1}^T(\lambda_{i,t+1}\!-\!\lambda)=\sum_{i=1}^{N}y_{i,t+1}^T(\lambda_{i,t+1}\!-\!\bar{\mu}_t\!+\!\bar{\mu}_t\!-\!\lambda) \displaybreak[0] \notag\\
		&\le B_y\sum_{i=1}^{N}\|\bar{\mu}_t\!-\!\lambda_{i,t+1}\|+\!\sum_{i=1}^N g_i(x_{i,t+1})^T(\bar{\mu}_t\!-\!\lambda)\displaybreak[0] \notag\\ 
		&\!\le\!(\!B_y\!\!+\!\!F)\!\sum_{i=1}^{N}\!\|\bar{\mu}_t\!\!-\!\!\lambda_{i,t+1}\|\!\!+\!\!\sum_{i=1}^N\! g_i(x_{i,t+1})^T\!(\lambda_{i,t+1}\!\!-\!\!\lambda), \label{eq:yitamultiplylambdait1minuslambdabound}
	\end{align}
	where the last inequality utilizes Lemma~\ref{lem:sumyit=sumgit} and \eqref{eq:bounofgivalue}.
	Let $S_{i,t}(x_i,\lambda_i)=\alpha_t\partial f_{i,t}(x_{i,t})^T\!(x_i\!-\!x_{i,t})+\langle \lambda_{i}, ~g_i(x_i)\rangle\!+\eta_t\|x_i\!-\!x_{i,t}\|^2 $.
Obviously, we have $\sum_{i=1}^N g_i(x_{i,t+1})^T(\lambda_{i,t+1}\!-\!\lambda)=\sum_{i=1}^NS_{i,t}(x_{i,t+1},\lambda_{i,t+1})-S_{i,t}(x_{i,t+1},\lambda) \le \sum_{i=1}^N\!S_{i,t}(\tilde{x}_{i,t}, \lambda_{i,t+1})\!\!-\!\!S_{i,t}(x_{i,t+1},\lambda)\!\!-\!\!\eta_t\|x_{i,t+1}\!\!-\!\!\tilde{x}_{i,t}\|^2$, $\forall \tilde{x}_{i,t} \in X_i$, which follows from the $2\eta_t$-strong convexity of $S_{i,t}(x_i,\lambda_{i,t+1})$ with respect to the variable $x_i$. Combing the inequality with $\sum_{i=1}^N\!S_{i,t}(\tilde{x}_{i,t}, \lambda_{i,t+1})\!-\!S_{i,t}(\tilde{x}_{i,t}, \bar{\mu}_t) \le F\!\sum_{i=1}^{N}\!\|\bar{\mu}_t\!-\!\lambda_{i,t+1}\|$ yields
	\begin{align}
		&\sum_{i=1}^N g_i(x_{i,t+1})^T(\lambda_{i,t+1}\!-\!\lambda)\le F\sum_{i=1}^{N}\!\|\bar{\mu}_t\!-\!\lambda_{i,t+1}\|\displaybreak[0]\notag\\
		&+\!\sum_{i=1}^N\!S_{i,t}(\tilde{x}_{i,t}, \bar{\mu}_t)\!-\!\!S_{i,t}(x_{i,t+1},\lambda)\!\!-\!\eta_t\|x_{i,t+1}\!\!-\!\!\tilde{x}_{i,t}\|^2. \displaybreak[0] \label{eq:sumgxit1timeslambdait1bound}
	\end{align} 
	
	Let $\lambda=\mathbf{0}_p$. Imitating the Lemma~4 in \cite{c17} leads to
	$-S_{i,t}(x_{i,t+1},\lambda)=-\big(\alpha_t\sum_{i=1}^{N}\partial f_{i,t}(x_{i,t})^T\!(x_{i,t+1}\!-\!x_{i,t})\!+\!\eta_t\sum_{i=1}^{N}\|x_{i,t+1}\!-\!x_{i,t}\|^2\big) \le \frac{NG^2\alpha_t^2}{4\eta_t}$. By combing this inequality with \eqref{eq:barmut1withepsilonbarmut}-\eqref{eq:sumgxit1timeslambdait1bound}, dividing both sides by $\frac{2}{N}$, and substituting the
	expressions of $S_{i,t}(\tilde{x}_{i,t}, \bar{\mu}_t)$ yield \eqref{eq:barmut1barmutnormbound}. Thus, Lemma~\ref{lem:mubart1normandmubartbound} holds.

	\subsection{Proof of Theorem~\ref{thm:TVdynamicregret}} \label{prf:TVdynamicregret}
	For any $t \ge 1$, let $\tilde{x}_{i,t}=x_{i,t}^*$, $\forall i \in \mathcal{V}$, which results in $\sum_{i=1}^{N}g_i(x_{i,t}^*) \le \mathbf{0}_p$. With $\bar{\mu}_t \ge \mathbf{0}_p$, $\langle\bar{\mu}_t, \sum_{i=1}^{N}g_i(x_{i,t}^*)\rangle \le 0$. By virtual of the convexity of $f_{i,t}$, we have $\sum_{i=1}^{N}\!\alpha_t\partial f_{i,t}(x_{i,t})^T\!(x_{i,t}^*\!-\!x_{i,t})\! \le\! \alpha_t\!\sum_{i=1}^{N}\!f_{i,t}(x_{i,t}^*)\!-\!f_{i,t}(x_{i,t})$. Equipped with these, we divide \eqref{eq:barmut1barmutnormbound} both sides by $\alpha_t$ and  then sum it from $t=1$ to $T$ to obtain
	\begin{align}
		&\sum_{t=1}^{T}\sum_{i=1}^{N} f_{i,t}(x_{i,t})-\sum_{t=1}^{T}\sum_{i=1}^{N} f_{i,t}(x_{i,t}^\star) \le\underbrace{(\frac{N}{2}+\frac{N}{r})\sum_{t=1}^{T}\frac{B_y^2}{\alpha_t}}_{S_1} \displaybreak[0] \notag\\
		&+\underbrace{\sum_{t=1}^{T}\frac{NG^2\alpha_t}{4\eta_t}}_{S_2}+ \frac{N}{2}\underbrace{\sum_{t=1}^{T}\frac{1}{\alpha_t}(\|\bar{\mu}_{t}\|^2-\|\bar{\mu}_{t+1}\|^2)}_{S_3} \displaybreak[0] \notag\\
		&+(2B_y\!+\!2F)\underbrace{\sum_{t=1}^{T}\frac{1}{\alpha_t}\!\sum_{i=1}^{N}\|\bar{\mu}_t\!-\!\lambda_{i,t+1}\|}_{S4} \displaybreak[0] \notag\\
		&+\underbrace{\!\sum_{t=1}^{T}\frac{\eta_t}{\alpha_t}\sum_{i=1}^{N}\left(\|x_{i,t}^\star-x_{i,t}\|^2-\|x_{i,t+1}\!-\!x_{i,t}^\star\|^2\right)}_{S_5}.\label{eq:VtsumtTfitfitstar}
	\end{align}
		Below, we analyze the upper bounds of $S_i$, $i =1, \ldots, 5$. With $\alpha_t=\sqrt{t}$ and $\eta_t=t$, it is easy to obtain
		\begin{align}
			&S_1\le	(NB_y^2+\frac{2NB_y^2}{r})\sqrt{T},~S_2\le\frac{NG^2\sqrt{T}}{2}, \displaybreak[0] \label{eq:sumt1divideVt}\\
			&S_3\!=\!\|\bar{\mu}_{1}\|^2\!\!+\!\!\sum_{t=2}^{T}\!(\frac{1}{\alpha_t}\!-\!\frac{1}{V_{t-1}})\|\bar{\mu}_{t}\|^2\!\!-\!\frac{1}{\alpha_t}\|\bar{\mu}_{T+1}\|^2 \le 0, \displaybreak[0] \label{eq:sumtVtbarmutbarmut1}
		\end{align} 
		where \eqref{eq:sumt1divideVt} follows from $\sum_{t=1}^{T} \frac{1}{\sqrt{t}} \le1+\int_{t=1}^{T}t^{-1/2}dt\le 2\sqrt{T}$ and \eqref{eq:sumtVtbarmutbarmut1} is because $\|\bar{\mu}_{1}\|^2=0$ and $\frac{1}{\alpha_t}-\frac{1}{\alpha_{t-1}} \le 0$.
		From Lemma~\ref{lem:dualconsesusboud}, we have
		\begin{align}
			S_4\le\! \frac{8N^2B_y\sqrt{p}}{r}\!\sum_{t=1}^{T}\frac{1}{\alpha_t}\!\sum_{k=1}^{t}\sigma^{t-k}\le\frac{16N^2B_y\sqrt{p}\sqrt{T}}{r(1-\sigma)}, \label{eq:sumTconsensuserror}
		\end{align}
		where the last inequality in \eqref{eq:sumTconsensuserror} comes from $\!\sum_{t=1}^{T}\frac{1}{\alpha_t}\!\sum_{k=1}^{t}\sigma^{t-k}\le \sum_{t=0}^{T-1}\sigma^{t}\sum_{k=1}^{T}\frac{1}{\alpha_k}$. Let $\mathbf{x}_{t}=[(x_{1,t})^T, \ldots, (x_{N,t})^T]^T$. Similar to the proof of Theorem~2 in \cite{c19}, the term $S_5$ is rewritten as
		\begin{align}
			&S_5\le\|\mathbf{x}_{1}\!-\!\mathbf{x}_1^\star\|^2+\sum_{t=1}^{T}(\sqrt{t+1}-\sqrt{t})\|\mathbf{x}_{t+1}\!-\!\mathbf{x}_{t+1}^\star\|^2\displaybreak[0]\notag\\
			&+\sum_{t=1}^{T}\sqrt{t}(\mathbf{x}_{t+1}^\star\!-\!\mathbf{x}_{t}^\star)^T(\mathbf{x}_{t+1}^\star+\mathbf{x}_{t}^\star-2\mathbf{x}_{t+1})\displaybreak[0] \notag \\
			&\le NR^2(1\!+\!\sum_{t=1}^{T}(\sqrt{t+1}\!-\!\sqrt{t})) \!+\!2NR\!\!\sum_{t=1}^{T}\!\sqrt{t}\|\mathbf{x}_{t+1}^\star\!-\!\mathbf{x}_{t}^\star\|\displaybreak[0]\notag\\
			&\le 2NR^2\sqrt{T}+2NRV_T,\label{eq:etatVtsumtxtxstar}
		\end{align}
		where $V_T:=\sum_{t=1}^T\sqrt{t}\sum_{i=1}^N\left\|x_{i, t+1}^*-x_{i, t}^*\right\|$ and the last inequality follows from  Assumption~\ref{asm:primalcouplingprob} and $\sqrt{T+1}\le 2\sqrt{T}$, $\forall T \ge 1$.
		Combing \eqref{eq:VtsumtTfitfitstar} with \eqref{eq:sumt1divideVt}--\eqref{eq:etatVtsumtxtxstar} yields Theorem~\ref{thm:TVdynamicregret}.
		
		\subsection{Proof of Theorem~\ref{thm:WioutSlaterCV}} \label{prf:WioutSlaterCV}
		By $\tilde{x}_{i,t}=\tilde{x}_{i}$, $\forall i \in \mathcal{V}$ that satisfies $\sum_{i=1}^N g_i(\tilde{x}_i)\le \mathbf{0}_p$, we have $\langle\bar{\mu}_t, \sum_{i=1}^{N}g_i(\tilde{x}_{i})\rangle \le 0$. Based on this and $\alpha_t=\sqrt{t}$, $\eta_t=t$, summing \eqref{eq:barmut1barmutnormbound} from $t=1$ to $T$ yields  
		\begin{align}
			&\frac{N}{2}\sum_{t=1}^T(\|\bar{\mu}_{t+1}\|^2-\|\bar{\mu}_{t}\|^2) =\frac{N}{2}\|\bar{\mu}_{T+1}\|^2\le\!(\frac{N}{2}\!\!+\!\!\frac{N}{r})TB_y^2\!\displaybreak[0] \notag\\
			&+\sum_{t=1}^T\!\frac{NG^2\alpha_t^2}{4\eta_t}\!+\!(2B_y\!+\!2F)\!\sum_{t=1}^T\sum_{i=1}^{N}\|\bar{\mu}_t\!-\!\lambda_{i,t+1}\|\displaybreak[0]\notag\\
			&+\sum_{t=1}^T\sum_{i=1}^N \alpha_t\partial f_{i,t}(x_{i,t})^T\!(\tilde{x}_{i,t}\!-\!x_{i,t})\!\displaybreak[0] \notag\\ 
			&+\sum_{t=1}^T\sum_{i=1}^N \eta_t(\|x_{i,t}\!-\tilde{x}_{i,t}\|^2\!\!-\!\|x_{i,t+1}\!-\!\tilde{x}_{i,t}\|^2)\displaybreak[0] \notag\\ 
			&\le  (\frac{N}{2}\!\!+\!\!\frac{N}{r})TB_y^2\!+\frac{NG^2T}{4}+\!(2B_y\!+\!2F)\!\frac{8N^2B_y\sqrt{p}T}{r(1-\sigma)}\displaybreak[0]\notag\\
			&+NGRT^{\frac{3}{2}}+2TNR^2,\displaybreak[0]
			\label{eq:muT1normbound}
		\end{align}
		where Lemma~\ref{lem:dualconsesusboud}, Cauchy–Schwarz inequality, Assumption~\ref{asm:primalcouplingprob}, \eqref{eq:boundsibradientoffigi}, \eqref{eq:etatVtsumtxtxstar}, and $\sum_{t=1}^{T}\alpha_t \le 1+\int_{t=1}^{T}t^{1/2}dt\le T^{\frac{3}{2}}$ are used to infer the last inequality.
		In light of \eqref{eq:updateofcouplinguk}, we have $\mu_{i,t+1} \ge \hat{\mu}_{i, t}+y_{i,t+1}$. Summing this inequality from $i=1$ to $N$ gives $\bar{\mu}_{t+1} \ge \bar{\mu}_{t}+\frac{1}{N}\sum_{i=1}^N g_i(x_{i,t+1})$, which leads to
		$\sum_{t=1}^T\sum_{i=1}^N g_{i}(x_{i,t+1})\le N\sum_{t=1}^T(\bar{\mu}_{t+1}-\bar{\mu}_{t})\le N\bar{\mu}_{T+1} \le N\|\bar{\mu}_{T+1}\|$. Invoking to the convexity of $g_i$ gives $
		\sum_{t=1}^T\sum_{i=1}^N g_{i}(x_{i,t})\le \sum_{t=1}^T\sum_{i=1}^N g_{i}(x_{i,t+1})+NGR\mathbf{1}_p\le N\bar{\mu}_{T+1}+NGR\mathbf{1}_p$,
		which leads to
		\begin{align}
			\operatorname{Reg}^c(T)	\le N\|\bar{\mu}_{T+1}\|+NGR\sqrt{p}.\displaybreak[0] \label{eq:sumTsumNgitgit1}
		\end{align}
		The inequality  \eqref{eq:muT1normbound} implies $\|\bar{\mu}_{T+1}\|=\mathcal{O}(T^{\frac{3}{4}})$. Inserting it with \eqref{eq:sumTsumNgitgit1} gives Theorem~\ref{thm:WioutSlaterCV}.
		
		\subsection{Proof of Theorem~\ref{thm:WiSlaterCV}} \label{prf:WiSlaterCV}
		From \eqref{eq:sumTsumNgitgit1}, we observe that $\operatorname{Reg}^c(T)$ depends on $\|\bar{\mu}_{T+1}\|$. The following lemma presents a smaller bound of $\|\bar{\mu}_{T+1}\|$ than \eqref{eq:muT1normbound}, enabling a smaller bound on $\operatorname{Reg}^c(T)$. 
		\begin{lemma}\label{lem:withSCboundofbarmut}
			Let $\tau=\lceil \sqrt{t} \rceil$, $\delta=B_y+\epsilon$. For any $t \ge 1$,
			\begin{align}
				\|\bar{\mu}_t\|&\le  4\delta\sqrt{t}\!+\!	\theta_{t}(\tau)\!+\!\frac{16\sqrt{t}\delta^2}{\epsilon} \log\frac{32\delta^2}{\epsilon^2}\!+\!6B_y. \displaybreak[0]
			\end{align}
			where $\theta_t(\tau)=(1+\frac{2}{r})\frac{B_y^2}{\epsilon}+\frac{G^2}{2\epsilon}\!+\!\frac{(2B_y\!+\!2F)16NB_y\sqrt{p}}{r\epsilon(1-\sigma)}+\frac{4GR\alpha_t}{\epsilon}+\frac{4R^2\eta_t}{\epsilon\tau}+(2 B_y^2+\epsilon)\tau$.
		\end{lemma}
		\begin{proof}
			We first bound the difference between $\|\bar{\mu}_{t+1}\|$ and $\|\bar{\mu}_{t}\|$, $\forall t \ge 1$, i.e.,
			\begin{align}
				-B_y\le \|\bar{\mu}_{t+1}\|- \|\bar{\mu}_{t}\|\le B_y,\displaybreak[0] \label{eq: barmut1barmutlowuppbound}	
			\end{align}
			where \eqref{eq:boundofepsiloncit} and \eqref{eq:barmut1barmut} give rise to the right-hand inequality. With regard to the left-hand inequality, it can be obtained $\|\bar{\mu}_{t}\|-\|\bar{\mu}_{t+1}\|\le \|\bar{\mu}_{t+1}-\bar{\mu}_{t}\|=\|\frac{1}{N}\sum_{i=1}^{N}\epsilon_{i,t+1}\|\le B_y$.
			
			Let $\tilde{x}_{i,t}=\hat{x}_{i}$ and $\triangle_s=\frac{1}{2}\|\bar{\mu}_{s+1}\|^2-\frac{1}{2}\|\bar{\mu}_{s}\|^2$, Summing \eqref{eq:barmut1barmutnormbound} from $s=t, t+1, \ldots, t+\tau-1$, we have
			\begin{align}
				&\sum_{s=t}^{t+\tau-1}\!\!\triangle_s\le\!(\frac{1}{2}\!+\!\frac{1}{r})B_y^2\tau\!+\!\frac{G^2\tau}{4}\!+\!\eta_{t+\tau-1}R^2-\!\epsilon\!\!\sum_{s=t}^{t+\tau-1}\!\!\|\bar{\mu}_{s}\|\!\displaybreak[0]\notag\\
				&+\frac{2B_y\!+\!2F}{N}\!\sum_{s=t}^{t+\tau-1}\!\!\sum_{i=1}^{N}\!\|\bar{\mu}_s\!-\!\lambda_{i,s+1}\|+\!GR\sum_{s=t}^{t+\tau-1}\!\alpha_s\!,	\displaybreak[0]\label{eq:trianglesbound}
			\end{align}
			where $\eta_{t+\tau-1}R^2$ is obtained by referring to \eqref{eq:etatVtsumtxtxstar}. Based on Assumption~\ref{asm: slatercondition}, the term $-\epsilon\sum_{s=t}^{t+\tau-1}\|\bar{\mu}_s\|$ comes from  $\sum_{s=t}^{t+\tau-1}\langle\bar{\mu}_s,\sum_{i=1}^{N}g_i(\hat{x}_{i})\rangle\le \sum_{s=t}^{t+\tau-1}\langle\bar{\mu}_s,\epsilon\mathbf{1}_p\rangle\le  -\epsilon\sum_{s=t}^{t+\tau-1}\|\bar{\mu}_s\|$, where $\bar{\mu}_s \ge \mathbf{0}_p$. Since $1 \le \tau \le t+1$ and $V_s=\sqrt{s}$, we obtain
			$\sum_{s=t}^{t+\tau-1}\alpha_s\le 2\tau \alpha_{t}$ and $\eta_{t+\tau-1}\le 2\eta_t$. By resorting to
			Lemma~\ref{lem:dualconsesusboud} and \eqref{eq: barmut1barmutlowuppbound}, we obtain
			\begin{align}
				&\sum_{s=t}^{t+\tau-1}\sum_{i=1}^{N}\|\bar{\mu}_s\!-\!\lambda_{i,s+1}\|\ \le \frac{8N^2B_y\sqrt{p}}{r(1-\sigma)}\tau, \displaybreak[0] \label{eq:sumTaudualconsenus}\\
				&\sum_{s=t}^{t+\tau-1}\|\bar{\mu}_{s}\|\!\!\ge\!\!\! \sum_{s=t}^{t+\tau-1}\!(\|\bar{\mu}_{t}\|\!\!-\!\!(s\!-\!t)B_y)\!\ge\! \tau \|\bar{\mu}_{t}\|\!-\!\tau^2B_y, \displaybreak[0] \label{eq:barmuslowbound}
			\end{align}
			which together with 
			\eqref{eq:trianglesbound} results in
			\begin{align}
				&\sum_{s=t}^{t+\tau-1}\triangle_s\le(\frac{1}{2}\!+\!\frac{1}{r})B_y^2\tau\!+\!\frac{G^2\tau}{4}\!+\!2R^2\eta_t\!+\!2GR\tau \alpha_t\displaybreak[0]\notag\\
				&+\!\frac{(2B_y\!+\!2F)8NB_y\sqrt{p}}{r(1-\sigma)}\tau\!+\!\epsilon\tau^2B_y\!-\!\epsilon\tau \|\bar{\mu}_{t}\|. \displaybreak[0]\notag \label{eq: sumtrianglewithminusbarmut}
			\end{align}
			This inequality implies
			$\|\bar{\mu}_{t+\tau}\|^2\!=\!\|\bar{\mu}_{t}\|^2\!+\!2\sum_{s=t}^{t+\tau-1}\!\!\triangle_s \le\|\bar{\mu}_{t}\|^2\!-\!2\epsilon\tau \|\bar{\mu}_{t}\|+\epsilon\tau\theta_t(\tau)$ according to the definition of $\theta_t(\tau)$.
			Thus, if  $\|\bar{\mu}_t\|\ge \theta_t(\tau)$, we have
			\begin{equation}\label{eq:batmutaddtboundbytau}
				\|\bar{\mu}_{t+\tau}\|-\|\bar{\mu}_t\| \leq -\frac{\epsilon\tau}{2}, \forall t \ge 1.
			\end{equation}
			
			Next we utilize \eqref{eq:batmutaddtboundbytau} to bound $\|\bar{\mu}_t\|$. Consider the case $t \ge 6$. Let $\delta=B_y+\epsilon,~\xi=\frac{\epsilon}{2}$, $\tilde{r}=\frac{\xi}{4\lceil \sqrt{t}\rceil\delta^2}$, and $\rho=1-\frac{\tilde{r}\xi\tau}{2}$, which implies $0<\rho <1$. Denote $w_t=\|\bar{\mu}_t\|-\|\bar{\mu}_{t-\tau}\|$. According to \eqref{eq: barmut1barmutlowuppbound}, $w_t=\sum_{s=t-\tau}^{t-1}\|\bar{\mu}_{s+1}\|-\|\bar{\mu}_{s}\| \le \tau B_y \le \tau\delta $. Like Lemma~6 in \cite{c18},
			$e^{\tilde{r}\|\bar{\mu}_t\|}=e^{\tilde{r}(w_t+\|\bar{\mu}_{t-\tau}\|)}\le e^{\tilde{r}\|\bar{\mu}_{t-\tau}\|}(1+\tilde{r} w_t+\frac{1}{2} \tilde{r} \tau \xi)$. Note that $t-\tau \ge 1$, $\forall t \ge 6$. If $\|\bar{\mu}_{t-\tau}\|\ge \theta_{t-\tau}(\tau)$, we have $w_t=\|\bar{\mu}_t\|\!-\!\|\bar{\mu}_{t-\tau}\|\le-\frac{\epsilon\tau}{2}= -\xi\tau $ by \eqref{eq:batmutaddtboundbytau}, which implies
			\begin{align}
				e^{\tilde{r}\|\bar{\mu}_t\|} \le \rho e^{\tilde{r}\|\bar{\mu}_{t-\tau}\|}+e^{\tilde{r}\tau\delta}e^{\tilde{r}\theta_{t-\tau}(\tau)}.\displaybreak[0] \label{eq:erbarmutbound}
			\end{align}
			It is easy to verify that \eqref{eq:erbarmutbound} holds if $\|\bar{\mu}_{t-\tau}\|< \theta_{t-\tau}(\tau)$.
			Moreover, $\forall t \ge 6$, $\lfloor \frac{t}{\tau}\rfloor= k$ for some $k \ge 2$. Consequently,
			$t-(k-2)\tau \ge 1$. Thus, we can apply \eqref{eq:erbarmutbound} for $s=t, t-\tau, \ldots, t-(k-2)\tau $ to obtain
			\begin{align}
				&e^{\tilde{r}\|\bar{\mu}_t\|}\le \rho e^{\tilde{r}\|\bar{\mu}_{t-\tau}\|}+e^{\tilde{r}\tau\delta}e^{\tilde{r}\theta_{t-\tau}(\tau)}\displaybreak[0] \notag\\
				& \le \rho^{k-1} e^{\tilde{r} \|\bar{\mu}_{t-(k-1) \tau}\|}+e^{\tilde{r} \delta \tau} \sum_{i=1}^{k-1} \rho^{i-1} e^{\tilde{r} \theta_{t-i \tau}}\displaybreak[0] \notag\\
				&\le\rho^{k-1} e^{2\tilde{r}\tau\delta}+e^{\tilde{r} \delta \tau}e^{\tilde{r} \theta_{t}} \sum_{i=1}^{k-1} \rho^{i-1} \le \frac{e^{2\tilde{r}\tau\delta}e^{\tilde{r} \theta_{t}}}{1-\rho}, \displaybreak[0] \label{eq:tmore9erbarmutbound}
			\end{align}
			where the third inequality was resulted from: (1) $\|\bar{\mu}_{t-(k-1) \tau}\|\le (t-(k-1) \tau)B_y\le 2\tilde{r}\tau\delta$ according to \eqref{eq: barmut1barmutlowuppbound} and $t-(k-1) \tau \le 2\tau$; (2) $0 <\theta_{t-i \tau} \le\theta_{t}$  
			because $\theta_{t}$ increases with $t$.
			From \eqref{eq:tmore9erbarmutbound} and $\tau=\lceil\sqrt{t}\rceil\le 2\sqrt{t}$, we have
			\begin{align}
				\|\bar{\mu}_t\|&\le 2\tau\delta+\theta_{t}(\lceil\sqrt{t}\rceil)+\frac{1}{\tilde{r}} \log\frac{1}{1-\rho}\displaybreak[0] \notag\\
				& \le 4\delta\sqrt{t}+\theta_{t}(\lceil\sqrt{t}\rceil)\!+\!\frac{16\sqrt{t}\delta^2}{\epsilon} \log\frac{32\delta^2}{\epsilon^2}\!+\!6B_y. \displaybreak[0]\label{eq:withSCbarmutbound}
			\end{align}
			Consider the case $t<6$. It is straightforward to obtain $\|\bar{\mu}_t\|\le tB_y\le 6B_y$. Thus, Lemma~\ref{lem:withSCboundofbarmut} holds. 
		\end{proof}
		
		Since $\theta_{t}(\lceil\sqrt{t}\rceil) =\mathcal{O}(\sqrt{t})$ according to the definition of $\theta_{t}$ in Lemma~\ref{lem:withSCboundofbarmut}, by combing it with  \eqref{eq:withSCbarmutbound}, $\|\bar{\mu}_{t}\|=\mathcal{O}(\sqrt{t})$. Like \eqref{eq:sumTsumNgitgit1}, we have $\operatorname{Reg}^c(T)=\mathcal{O}(\sqrt{T})$. Thus, Theorem~\ref{thm:WiSlaterCV} holds.

		\addtolength{\textheight}{-3cm}   
		

	\end{document}